\documentclass[12pt]{article}
\usepackage{preprint-layout} 
\usepackage{preprint-notation} 
\usepackage[]{preprint-tools} 

\title{CutFEM Based on Extended Finite Element Spaces} 
\date{\today}

\author{
Erik Burman,
Peter Hansbo,
Mats G. Larson 
}
\date{}

\begin{document}

\maketitle

\begin{abstract}
We develop a general framework for construction and analysis of discrete extension operators with application to unfitted finite 
element approximation of partial differential equations. In unfitted methods so called cut elements intersected by the boundary 
occur and these elements must in general by stabilized in some way. Discrete extension operators provides such a stabilization by
modification of the finite element space close to the boundary. More, precisely the finite element space is extended from the stable 
interior elements over the boundary in a stable way which also guarantees optimal approximation properties. Our framework is applicable 
to all standard nodal based finite elements of various order and
regularity. We develop an abstract theory for elliptic problems and
associated parabolic time dependent partial differential
equations and derive a priori error estimates. We finally apply this
to some examples
of partial differential equations of different order
including the interface problems, the biharmonic operator and the
sixth order triharmonic operator.
\end{abstract}
\section{Introduction}
Building on ideas around unfitted finite element methods using
Nitsche's method and boundary stabilization terms from the papers
\cite{HH02,HHL03,HH04,BBH09, Bu10,BH12,BH14,HLZ14,MLLR14}, the framework of cut finite element methods
was proposed in \cite{BCHLM15}. The main idea was to use computational meshes
and finite element spaces
that were independent of the geometry of the physical problem. Both
the partial differential equation and the geometry were then
defined through the finite element variational form. In particular,
Nitsche's method was used to impose boundary and interface conditions
and certain stabilization terms, so called, ghost penalty terms
ensured stability even in the presence of unfavourable mesh interface
intersections.
Typically such ghost penalty terms acted on jumps of functions over
element edges or penalized the difference between polynomial
projections or extensions in the interface zone. This provides a
convenient solution for standard $C^0$, low order elements, but 
in other situations it may be less attractive.
For instance, when the element order becomes high, elements with higher regularity are used, 
or systems of equations with several different
differential operators are considered the design and evaluation of the
ghost penalty terms becomes non-trivial and costly. For such
situations we propose a different approach in this work. While still
assuming that the mesh is independent of the geometry we let
the approximation space be geometry dependent. To build in stability 
we use discrete extensions from the interior of the domain to interface
elements, that are at risk of leading to unstable linear systems in
the pde-discretization. We then prove that similar stability as for the ghost 
penalty can be obtained with optimal approximation properties and without the
introduction of any numerical parameters. Similar ideas have been exploited in 
agglomeration approaches for $C^0$ approximation spaces 
\cite{HWX17,BMV18a,BMV18b} and in discontinuous Galerkin methods \cite{JL13,BE18}. 
In this work we propose a complete framework for discrete extensions 
for finite element spaces that allow the treatment of nodal based elements 
of all orders and all regularities, with a rigorous analysis of stability and
approximation properties. 

The proposed CutFEM with discrete extension is particularly appealing for high
order elliptic problems, where suddenly elements that are accurate and
easy to construct, such as the Bogner-Fox-Schmit (BFS) element that is
$C^1$ or similar spaces with higher regularity, become
interesting. Typically the main drawback of such spaces that are based
on tensorization is that they can not fit the physical boundary. This
problem is solved in the CutFEM framework and the use of discrete
extension reduces the need of stabilization terms. We refer to
\cite{BHL20a} for details on previous work combining CutFEM and the
BFS-element. See \cite{EmbDol10} for related developments for second and 
fourth order problems and \cite{HS12} for the thin plate equation. The discrete 
extension is more intrusive than the penalty
approach, since the approximation space is modified, however once 
implemented it gives a flexible and robust tool
for CutFEM discretization methods. For a discussion of the use of
discrete extensions to allows for fully explicit timestepping for the
wave equation discretized using CutFEM we refer to \cite{BHL20b}.

\paragraph{{\bf{Outline:}}}
In Section \ref{sec:ext_FEM} we introduce the framework for
extended finite element spaces and prove some fundamental stability
and approximation results.
In Section \ref{sec:abstract} we discuss an abstract framework for cut
finite element methods using discrete extension spaces and Nitsche's
method. Optimal a priori error estimates
are derived using the properties of the extended space. The extension to parabolic problems 
is also briefly covered.  In Section \ref{sec:appli} we show that some important 
problem classes, such as fictitious domain problem or interface
problems subject to elliptic operators of second and fourth order fit naturally 
in the framework. Finally, in Section \ref{sec:num} some numerical
illustrations are presented.
\section{Extended Finite Element Spaces}\label{sec:ext_FEM}

In this section we develop an abstract theory for construction of extension operators for application to various types of unfitted 
discretizations of partial differential equations. We consider a conforming setting where the finite element space $V_h$ is a 
subspace of $H^l$, typically used to discretize an elliptic operator of order $2l$. The extension is constructed as the composition 
of an average operator that maps a discontinuous space $W_h$ onto $V_h$ and an extension in $W_h$ from interior elements to elements 
intersecting the boundary. Since $W_h$ is discontinuous we can easily extend from one element to another by canonical extension 
of polynomials.

\subsection{The Discrete Extension Operator}
\paragraph{The Mesh and Finite Element Space.}
\begin{itemize}
\item Let $\Omega$ be a domain in $\IR^d$ with boundary $\partial \Omega$ and let $\widetilde{\Omega}$ be a polygonal domain such that 
$\Omega \subset \widetilde{\Omega}$. 
\item Let $\widetilde{\mcT}_h$ be a quasiunform mesh on $\widetilde{\Omega}$ with mesh parameter $h \in (0,h_0]$ and define 
the active mesh $\mcT_h = \{ T \in \widetilde{\mcT}_h : T \cap \Omega \neq \emptyset \}$. Let $\Omega_h = \cup_{T \in \mcT_h} T$ be the domain 
covered by $\mcT_h$.
\item Let $\widetilde{V}_h$ be a finite element space on $\widetilde{\mcT}_h$ and let $V_h = \widetilde{V}_h |_{\mcT_h}$ be the 
active finite element space.
\end{itemize}

\paragraph{Definition of the Extension Operator.}
\begin{itemize}
\item Define the following partition of $\mcTh$,
\begin{equation}
\mcTh = \mcT_{h,B} \cup \mcT_{h,I}
\end{equation}
where $\mcT_{h,I}$ is the set of elements in the interior of $\Omega$  and $\mcT_{h,B}$ are the elements that intersect the boundary, 
\begin{equation}
\mcT_{h,I} = \{ T \in \mcT_h : T \subset \Omega\},\qquad \mcT_{h,B } = \mcTh \setminus \mcT_{h,I}
\end{equation}
Let $\Omega_{h,I} = \cup_{T \in \mcT_{h,I}} T$ and note that 
\begin{equation}
\Omega_{h,I} \subset \Omega \subset \Omega_h
\end{equation}
\item Let 
\begin{equation}
W_h = \mathbb{P}_k(\mcT_h) = \bigoplus_{T \in \mcT_h} \mathbb{P}_k(T)
\end{equation}
and let $V_h$ be a subspace of $W_h \cap H^l(\Omega_h)$.

\item Define the spaces 
\begin{equation}
W_{h,I} = W_h|_{\mcT_{h,I}}, \qquad V_{h,I} = V_h|_{\mcT_{h,I}}
\end{equation}
and let 
\begin{equation}
(\cdot )_I : W_h \ni w \mapsto (w)_I = w|_{\mcT_{h,I}} \in W_{h,I}
\end{equation}
denote the restriction of $w \in W_h$ to $\mcT_{h,I}$.

\item Define the extension operator 
\begin{equation}
\boxed{ E_h : V_{h,I} \ni v \mapsto A_h F_h v \in V_{h,I}^E \subset V_{h} }
\end{equation}
where $V_{h,I}^E = E_h V_{h,I}$ is the image of $V_{h,I}$ under the action of $E_h$, 
$A_h : W_h \rightarrow V_h$ is a linear average operator, and $F_h : W_{h,I} \rightarrow W_h$ 
is a linear extension operator and we define its image as the 
  subspace $W_h^E \subset W_h$,
\begin{equation}
W_h^E = \{w_h
\in W^h: w_h = F_h w_I, \, w_i \in W_{h,i} \}
\end{equation}
 We specify the properties of these operators below.
\end{itemize}

\paragraph{Norms.} For $l \in \IR$ we let $H^l(\omega)$ denote the standard Sobolev space of 
order $l$ with norm $\| \cdot \|_{H^l(\omega)}$ and semi norm $| \cdot |_{H^l(\omega)}$. For $l=0$ we use 
the notation $H^0(\omega) = L^2(\omega)$ and $\| \cdot \|_{L^2(\omega)} = \| \cdot \|_\omega$. For 
$l \in \mathbb{N}$ we define the following broken Sobolev norm on $W_h$,
\begin{align}
\| v \|^2_{H^l (\mcT_h)} 
&=
\sum_{j=0}^l \| \nabla^j v \|^2_{\mcT_h} 
\end{align}
where $\nabla^j v = \otimes_{k=0}^j \nabla v$ is the tensor of all $j$:th order partial derivatives of $v$, and 
$\|v \|^2_{\mcT_h} = \sum_{T \in \mcT_h} \| v \|^2_T$. For $l=0$ we {let 
 $\| v \|^2_{H^0(\mcT_h)} = \| v \|^2_{\mcT_h}$}.  We note that we can replace $\mcT_h$ by $\mcT_{h,I}$ and 
 $\mcF_{h}$ by $\mcF_{h,I}$ and get the corresponding discrete Sobolev norms on $W_{h,I}$.

\paragraph{Assumptions.}
\begin{description}
\item[A1] The space $W_h$ (and its subspaces $V_h$ and $W_h^E$) satisfies the inverse inequality 
\begin{equation}\label{eq:inverse-Wh}
\| w \|_{H^m(\mcT_h)} \lesssim h^{-m} \| w \|_{\mcT_h} 
\end{equation}
Here and below $\lesssim$ means less or equal up to a constant that is independent of the mesh 
parameter and the intersection of the domain and the mesh.

\item[A2] The spaces $V_h$ and $W_h^E$ satisfies the following approximation properties. For 
$v \in H^{s}(\Omega_h)$, with $0\leq s \leq k+1$ there is $v_\star \in V_h$ such that
\begin{align}\label{eq:approx-Vh}
\| v - v_\star \|_{H^m(\Omega_h)} \lesssim h^{s-m} | v |_{ H^{s}(\Omega_h) }
\end{align} 
and $w_\star \in W_h^E$ such that
\begin{align}\label{eq:approx-WhE}
\| v - w_\star \|_{H^m(\Omega_h)} \lesssim h^{s-m} | v |_{ H^{s}(\Omega_h) }
\end{align} 

\item[A3] The operator $A_h : W_h \rightarrow V_h$ is linear, bounded
\begin{equation}\label{eq:Ah-bounded}
\| A_h v \|_{\mcT_h} \lesssim \| v \|_{\mcT_h}
\end{equation}
 and 
\begin{equation} \label{eq:Ah-identity}
 A_h v = v \qquad v \in V_h
 \end{equation}
\item[A4] The operator $F_h : W_{h,I} \rightarrow W_h$ is linear, bounded
\begin{equation}\label{eq:Fh-bounded}
\| \nabla^j F_h w \|_{\mcT_h} \lesssim \| \nabla^j w \|_{\mcT_{h,I}}
,\qquad 0 \le j \le l
\end{equation}
and
\begin{equation}  \label{eq:Fh-identity}
 F_h (w)_I = w, \qquad w \in W_h^E
 \end{equation}
\end{description}

\subsection{Properties of the Extension Operator}
In this section we will show that any extension operator constructed
using an extension operator $F_h$ and and averaging operator $A_h$
satisfying the assumptions A1-A4, has
properties making it suitable for approximation using CutFEM.
We first show a stability estimate which typically is needed to establish coercivity of Nitsche's method and then 
we show that the extended finite element space $V_h^E$ has an optimal order approximation property.

\begin{lem}{\bf (Stability).}\label{lem:stab_E}
The extension operator $E_h: V_{h,I} \rightarrow V_h^E$, where $V_h^E \subset V_h \subset H^l(\Omega_h)$, 
satisfies the stability estimate
\begin{align}
\boxed{\| \nabla^j E_h v \|_{\Omega_h} \lesssim  \| \nabla^j v \|_{\Omega_{h,I}}, \qquad 0\leq j \leq l}
\end{align}
\end{lem}

\begin{proof} Adding and subtracting the identity operator, using the triangle inequality, and the 
inverse estimate (\ref{eq:inverse-Wh}) in A1, we obtain for an arbitrary $w \in W_h$,
\begin{align} \label{eq:A_stab_aux1}
\|\nabla^j A_h w\|_{\mcTh} &\leq \|\nabla^j (A_h - I ) w\|_{\mcTh} + \|\nabla^j  w\|_{\mcTh}
\\
&\lesssim h^{-j} \|(A_h - I) w\|_{\mcTh} + \|\nabla^j  w\|_{\mcTh}
\\
&\lesssim h^{-j} \|(A_h - I) (w - v') \|_{\mcTh} + \|\nabla^j  w\|_{\mcTh}
\\
&\lesssim h^{-j} \| w - v' \|_{\mcTh} + \|\nabla^j  w\|_{\mcTh} \label{eq:A_stab_aux2}
\end{align}
where we used (\ref{eq:Ah-identity}) in A3,  $(A_h - I) v =0$ for all $v \in V_h$, to insert an arbitrary $v' \in V_h$, and 
finally the $L^2$ boundedness (\ref{eq:Ah-bounded}) in A3 of $A_h$. Setting $w = F_h ( v )_I$ and $v' = v$ we get 
\begin{align}
\|\nabla^j A_h F_h (v )_I\|_{\mcTh} &\lesssim h^{-j} \| F_h (v )_I - v \|_{\mcTh} + \|\nabla^j F_h (v)_I \|_{\mcTh}
\\
&\lesssim h^{-j} \| F_h (v )_I - w_\star\|_{\mcTh} + h^{-j} \| w_\star - v \|_{\mcTh} + \|\nabla^j  F_h (v)_I\|_{\mcTh}
\\
&\lesssim h^{-j} \| F_h (v - w_\star)_I \|_{\mcTh} + h^{-j} \| w_\star - v \|_{\mcTh} + \|\nabla^j v \|_{\mcT_{h,I}}
\\
&\lesssim h^{-j} \| w_\star - v \|_{\mcTh} + \|\nabla^j v \|_{\mcTh}
\\
&\lesssim \| \nabla^j v \|_{\mcTh} + \| \nabla^j v \|_{\mcTh}
\end{align}
Here we added and subtracted $w_\star \in W_h^E$ and used the properties
of $F_h$ from A4. First the identity property (\ref{eq:Fh-identity}) and
then the 
 boundedness (\ref{eq:Fh-bounded}) . Finally the approximation property (\ref{eq:approx-WhE}) in A2 
is applied to $w_\star - v$ where we recall that $v \in H^j(\Omega_h)$
and we use the trivial bound $\|\nabla^j v \|_{\mcT_{h,I}}
\le \| \nabla^j v \|_{\mcTh}$.
\end{proof}

%

\begin{lem}{\bf (Approximation Property).}\label{lem:approx} For each $v \in H^{s}(\Omega_h)$, $0\leq s \leq k+1$, 
there is $v_\star^E \in V_h^E$ such that 
\begin{equation}\label{eq:approx-prop-VhE}
\boxed{\|v - v_\star^E \|_{H^m(\Omega_h)}  \lesssim h^{s-m} |v|_{H^{k+1}(\Omega)}, \qquad 0\leq m \leq l}
\end{equation}
where the hidden constant is independent of $v\in H^{s}(\Omega_h)$, the
intersection of $\partial \Omega$ with the elements, and the mesh size $h$.
\end{lem}
\begin{proof} We shall show that $v_\star^E = A_h w_\star$, with $w_\star$ as in (\ref{eq:approx-WhE}), satisfies 
(\ref{eq:approx-prop-VhE}). To that end, adding and subtracting 
$w_\star \in W_h \in W_h^E \subset H^l(\Omega_h)$ and using the
triangle inequality show that
\begin{equation}
\| v - A_h (v_\star)_I \|_{H^m(\Omega_h)}
\leq
\underbrace{\| v - w_\star \|_{H^m(\mcTh)}}_{I} + \underbrace{\| w_\star - A_h w_\star\|_{H^m(\mcTh)}}_{II} 
\end{equation}
\paragraph{Term $\bfI$.} Using the approximation property  (\ref{eq:approx-Vh}) for $V_h$ we directly have
\begin{equation}
\| v - w_\star \|_{H^m(\Omega_h)}  
\lesssim  h^{s-m} | v |_{H^s(\Omega_h)}
\end{equation}
\paragraph{Term $\bfI\bfI$.} Using the inverse inequality (\ref{eq:inverse-Wh}) to pass to the $L^2$-norm, adding and 
subtracting $v_\star \in V_h$, which satisfies (\ref{eq:approx-Vh}), and using the triangle inequality we obtain
\begin{align}
& \| w_\star  - A_h (w_\star) \|_{H^m(\mcTh)}   
\\ \label{eq:approx-a}
&\qquad \lesssim  h^{-m} \| w_\star  - A_h (v_\star)  \|_{\mcTh}   
\\
&\qquad \lesssim  h^{-m} \| v_\star  - w_\star - A_h (v_\star - w_\star )_I  \|_{\mcTh}   
\\
&\qquad \lesssim  h^{-m} \| v_\star  - w_\star\|_{\mcT_h}  + h^{-m} \| A_h (v_\star - w_\star )_I  \|_{\mcTh}   
 \\
&\qquad \lesssim  h^{-m} \| v_\star  - w_\star\|_{\mcT_h}  
 \\
 &\qquad \lesssim  h^{-m}\| v_\star  - v \|_{\mcT_h}  + h^{-m} \| v - w_\star\|_{\mcT_h}  
 \\ \label{eq:approx-b}
 &\qquad \lesssim 
 h^{s-m} |  v  |_{H^{s}(\Omega_h)} 
\end{align}
where we used the $L^2$ stability (\ref{eq:Ah-bounded}) of $A_h$ followed by the obvious fact that $\| (w)_I  \|_{\mcT_{h,I}} 
\leq \| w \|_{\mcT_{h,I}}$, then we added and subtracted $v$, and finally used the approximation properties 
(\ref{eq:approx-Vh}) and (\ref{eq:approx-WhE}) for $V_h$ and $W_h^E$. 
\end{proof}

\subsection{Construction of $\boldmath{A}_{\boldmath{h}}$}

In order to define a general average operator for nodal finite elements we recall the following 
definitions.
\begin{itemize}
\item For each element $T$ let $V_h(T)$ be the element finite element space. Let $\{ \varphi^*_{T,x} \}_{x\in \mcX_T}$ be a 
basis for the dual space $V^*_h(T)$ and let the corresponding Lagrange basis $\{\varphi_{T,x} \}_{\mcX_T}$ in $V_h(T)$ be 
defined by $\varphi^*_x(\varphi_y) = \delta_{xy}$ for $x,y \in \mcX_T$.   

\item Assume that the degrees of freedom are nodal degrees of the form $\varphi^*_{T,x}(v) = D^{\alpha_x}v(\xi_x)$, where $D^{\alpha_x}$ 
is a partial differential operator with multi index $\alpha_x $ and $\xi_x\in \IR^d$ is the physical node. Observe that  there can be several 
functionals $\varphi^*_{T,x}$,  associated to the element/node pair $(T,\xi_x)$, corresponding to different derivatives $D^{\alpha_x} v (\xi_x)$. 
We refer to the pair $x=(\alpha_x, \xi_x)$ as a generalized node.
  
\item Let $\mcT_h(x)$ be the set of elements $T \in \mcT_h$ such that 
$\xi_x \in \overline{T}$ and define the global basis function $\varphi_x$ at node $x$ by $(\varphi_x)|_T = \varphi_{T,x}$ for all 
$T\in \mcT_h(x)$. Let $\mcX_h$ be the set of all global nodes. The
functions $\{ \varphi_x \, | \, x \in \mcX_h\}$ form a basis for $V_h$. 
For properly constructed spaces we have $V_h \subset H^l(\Omega_h)$.
\end{itemize}

\paragraph{Definition.}
Let the nodal averaging operator $A_h:W_h \mapsto V_h$ be defined by
\begin{equation}\label{eq:averege-operator}
\boxed{A_h: W_h \ni w \mapsto \sum_{x \in \mcX_h}  \langle \varphi_x^*(w) \rangle_{x}  \varphi_x  \in V_h}
\end{equation}
where the average of the discontinuous function $w\in W_h$ at a node $x\in \mcX_h$ 
is defined by 
\begin{equation}\label{eq:average-nodal}
\langle \varphi_x^*(w) \rangle_x 
= 
 \sum_{T \in \mcT_h(x)}  \kappa_{T,x} \varphi_{T,x}^*(w_T)
\end{equation}
with  weights $\kappa_{T,x}$ satisfying
\begin{equation}
\kappa_{T,x}\geq 0, \qquad \sum_{T \in \mcT_h(x)}  \kappa_{T,x} = 1
\end{equation}

\begin{lem} The average operator $A_h$ defined by
  (\ref{eq:averege-operator}) satisfies assumption A3. As a consequence
\begin{equation}\label{eq:A_Hstab}
\|\nabla^j A_h v\|_{\mcTh} \lesssim h^{-j} \inf_{v' \in V_h} \| v - v'
\|_{\mcTh} + \|\nabla^j  v\|_{\mcTh},\quad 0 \leq j \leq m
\end{equation}
\end{lem}

\begin{proof} First we note that $A_h$ is linear, since the average $\langle \cdot \rangle_x$ 
is linear and $\varphi^*_{x,T}$ are linear functionals, and that $A_h$
is the identity on $V_h$ by 
construction. Next we note that the bound \eqref{eq:A_Hstab} was shown in Lemma
\ref{lem:stab_E} (inequalities \eqref{eq:A_stab_aux1} -\eqref{eq:A_stab_aux2}) under
the boundedness assumption (\ref{eq:Ah-bounded}). To verify the boundedness (\ref{eq:Ah-bounded}) 
we note that we have an equivalence of the form 
\begin{equation}\label{eq:equivalence-element}
\|v\|^2_T \sim \sum_{x\in \mcX_{h,T}} h^{d-|\alpha_x|} | \varphi^*_x(v)|^2  
\end{equation}
where $|\alpha_x|$ is the order of the differential operator $D^{\alpha_x}$. This equivalence follows 
by mapping to the reference element, application of equivalence of norms in finite dimension and
then mapping back to the physical element. We then have 
\begin{align}
\|A_h v\|^2_{\mcT_h} &= \sum_{T \in \mcT_{h}} \|A_h v\|^2_{T}  
\end{align}
and for each element contribution we use the equivalence (\ref{eq:equivalence-element}) followed by 
the Cauchy-Schwarz inequality
\begin{align}
\|A_h v\|^2_{T} 
&\lesssim   \sum_{x\in \mcX_{h,T}} h^{d-|\alpha_x|} \Big| \sum_{T \in \mcT_h(x)}  \kappa_{T,x} \varphi_{T,x}^*(v_T) \Big|^2
\\
& \lesssim  \sum_{x\in \mcX_{h,T}}  h^{d-|\alpha_x|} \Big( \sum_{T \in \mcT_h(x)}  \kappa^2_{T,x} \Big) 
\Big( \sum_{T \in \mcT_h(x)}  |\varphi_{T,x}^*(v_T)|^2 \Big) 
\\
&
 \lesssim \sum_{x\in \mcX_{h,T}}  h^{d-|\alpha_x|}
\Big( \sum_{T \in \mcT_h(x)}  |\varphi_{T,x}^*(v_T)|^2 \Big) 
\\
&
 \lesssim \sum_{x\in \mcX_{h,T}}  
 \sum_{T \in \mcT_h(x)} \Big(  h^{d-|\alpha_x|} |\varphi_{T,x}^*(v_T)|^2 \Big) 
\\
&
 \lesssim \sum_{x\in \mcX_{h,T}}  
 \sum_{T \in \mcT_h(x)}  \sum_{x\in \mcX_{h,T}} \Big(  h^{d-|\alpha_x|} |\varphi_{T,x}^*(v_T)|^2 \Big) 
\\
&
 \lesssim \sum_{x\in \mcX_{h,T}} \sum_{T \in \mcT_h(x)}  \| v \|^2_T
 \\
 &
 \lesssim   \| v \|^2_{\mcT_h(T)}
\end{align}
where we used shape regularity to conclude that there is a uniform bound on the number of elements 
sharing a node $x$. 
\end{proof}

\subsection{Construction of $\boldmath{F}_{\boldmath{h}}$}
\paragraph{Definition.}
Let $S_h: \mcT_{h,B}\rightarrow \mcT_{h,I}$ be a mapping that to each $T \in \mcT_{h,B}$ 
associates an element $T \in \mcT_{h,I}$  and assume that there is a constant such that for 
all $h \in (0,h_0]$ and $T \in \mcT_{h,B}$, 
\begin{equation}\label{eq:assumption}
\text{diam}(T\cup S_h(T)) \lesssim h
\end{equation}
where the hidden constant depends only on the shape regularity of
the mesh and the geometry of the interface. In fact we will show that for a Lipshitz domain there 
is such a mapping for $h_0$ small enough, see Lemma \ref{lem:Nh-map} below. 
We extend $S_h$ from $\mcT_{h,B}$ to $\mcT_h$ by letting $S_h(T) = T$ for $T\in \mcT_{h,I}$. 

For $v \in \mathbb{P}_k(T)$ we let $v^e \in \mathbb{P}_k(\IR^d)$ denote the canonical extension 
such that $v^e|_T = v$. We define the discrete extension operator $F_h : W_{h,I} \rightarrow W_h$ 
by
\begin{align}\label{eq:Fh}
\boxed{(F_h v )|_T = (v|_{S_h(T)})^e|_T} 
\end{align}

\paragraph{Macro Element Partition.}
Defining for each $T \in \mcT_h$ the extended element $T^E$, as the union of all elements $T'$ that are mapped 
to $T$ by $S_h$,
\begin{equation}
T^E = \cup_{T' \in S_h^{-1} (T)} T'
\end{equation}
and the resulting partition of $\Omega_h$,
\begin{equation}
\mcT_h^E = \{ T^E \, | \, T \in \mcT_h \}
\end{equation}
into macro elements $T^E$ of diameter $\text{diam}(T^E) \lesssim h$.  We also note that with this notation 
\begin{equation}
W_h^E = F_h (W_{h,I} ) = \bigoplus_{T^E \in \mcT_h^E} \mathbb{P}_k(T^E)
\end{equation}
i.e. the extension of $W_{h,I}$ is precisely the space of discontinuous piecewise polynomials on the 
macro element partition $\mcT_h^E$.

For the next lemma we recall, see Theorem 1.2.2.2 in \cite{Gr11}, that the Lipschitz property of the boundary is 
equivalent to following uniform cone property of the domain. Let $\text{Cone}_{\theta,\delta}(x)$ denote 
the open cone with vertex $x$, opening angle $\theta \in (0,\pi/2)$, and height $\delta>0$. The open domain 
$\Omega \subset \IR^d$ satisfies the  uniform cone property if for each $x \in \partial \Omega$ there is an open cone 
\begin{equation}
\text{Cone}_{\theta_0,\delta_0}(x) \subset \Omega
\end{equation}
with cone parameters $\theta_0$ and $\delta_0$ that are independent of $x$. 

\begin{lem} \label{lem:Nh-map} Assume that $\Omega \subset \IR^d$ is an open domain that satisfies the uniform cone property. 
Then for $h_0$ small enough there is a mapping $S_h:\mcTh \rightarrow \mcT_{h,I}$ that satisfies (\ref{eq:assumption}).
\end{lem}
\begin{proof} Take an element $T \in \mcT_{h,B}$ and let $x \in T \cap \partial \Omega$. By the uniform cone property there is a cone 
$\text{Cone}_{\theta_0,\delta_0}(x) \subset \Omega$ with opening angle $\theta_0\in (0,\pi/2)$ and height $\delta_0>0$. For $\delta \in (0,\delta_0)$ 
there is an open ball $B_r \subset \text{Cone}_{\theta_0,\delta}(x) \subset \text{Cone}_{\theta_0,\delta_0}(x)$ with radius $r \sim \delta $ since 
the opening angle $\theta_0 \in (0,\pi/2)$ is fixed. Now taking $c \delta =h$ for a sufficiently small constant it follows from quasi uniformity
that there is an element $\widetilde{T} \subset B_r$. This follows since if we consider the elements $\mcT_h(B_r)$ that intersect the ball $B_r$, then 
by shape regularity each such element is contained in a ball $B_{r'}$ with $r'\sim h$ and then the union of the balls $B_r'$ contains $B_r$. For $r' < r/2$ 
one of the balls $B_{r'}$ must be contained in $B_r$. Thus for all $h \in (0,h_0]$ with $h_0$ small enough, dependent on the quasiuniformity 
constants and the cone parameters $\theta_0$ and $\delta_0$, there is an element $\widetilde{T} \subset \Omega$ in a ball centred at $x$ 
with radius proportional to $h$, which concludes the proof.
\end{proof}

\begin{lem} {($F_h$ and $W_h^E$ satisfy assumptions A4 and A2).}
The operator $F_h$ defined by (\ref{eq:Fh}) satisfies
(\ref{eq:Fh-bounded}) and $W_h^E$ satisfies the approximation property (\ref{eq:approx-WhE}).
\end{lem}

\begin{proof} 
{First note that in view of (\ref{eq:assumption}), $\mcT_h^E = \{S_h^{-1}(T) : T \in \mcT_{h,I} \}$ is a partition of $\Omega_h$ into generalized 
elements all with diameter equivalent to $h$, more precisely, for each element $T^E \in \mcT_h^E$ there is a ball $B_{T^E,\delta}$ with diameter 
$\delta \sim h$ such that $T^E \subset B_{T^E,\delta}$. 

Then to verify that the stability (\ref{eq:Fh-bounded}) holds, we fix
$T^E \in \mcT_h^E $ such that for $T \in  \mcT_{h,I}$ $T^E =
S_h^{-1}(T)$. Note that due to shape regularity there is a ball $B_r(x)$ with radius 
$r \sim h$ and center $x$ such that $B_r(x) \subset T$. Using the
canononical extension $v^e \in \mathbb{P}_{k}(\mathbb{R}^d)$ such that
  $v^e\vert_T = v$ we have the inverse estimate $\| \nabla^j v^e \|_{B_{T^E,\delta}} \lesssim \| \nabla^j v \|_{B_r(x)}$. 
It follows that \begin{align*}
\| \nabla^j F_h w \|^2_{\mcT_h} & \lesssim \sum_{T^E \in \mcT_h^E}\|
\nabla^j  F_h w \|^2_{T^E} \lesssim \sum_{T^E \in \mcT_h^E}\|
\nabla^j  w \|^2_{B_{T^E,\delta}} \\
&  \lesssim \sum_{T^E \in \mcT_h^E}\|
\nabla^j  w \|^2_{B_{T^E,\delta}}\lesssim \sum_{T \in \mcT_{h,I}}\|
\nabla^j  w \|^2_{B_r(x)} \leq \sum_{T \in \mcT_{h,I}}\|
\nabla^j  w \|^2_{T} 
\end{align*}
Here we used the fact that the balls $B_{T^E,\delta}$ have finite overlap,
thanks to the shape regularity assumption.}

To verify  (\ref{eq:approx-WhE}) we recall that $T^E \subset B_{T^E,\delta}$
and directly employ the Bramble-Hilbert lemma, see \cite[Lemma 4.3.8]{BreSco}, 
to conclude that there 
is $w_{T^E} \in \mathbb{P}_k(B_{T^E,\delta})$ such that
\begin{equation}
\| v - w_{T^E} \|_{H^m(T^E)} \leq \| v - w_{T^E} \|_{H^m(B_{T^E,\delta})} \lesssim \delta^{s - m} \| v \|_{H^s(B_{T^E,\delta})}
\end{equation}  
Finally, summing over all the extended elements in $\mcT_h^E$ and
using the fact that $\delta \sim h$ and the finite overlap of the $B_{T^E,\delta}$, the approximation property 
(\ref{eq:approx-WhE}) for $W_h^E$ follows. 
\end{proof}

%

\begin{rem}\label{rem:eh} In practice, we can define the set of elements that have a large intersection 
with the domain as follows,
\begin{equation}
\mcT_{h,\text{large}} = \{ T \in \mcTh : |T \cap \Omega| \geq c h^d \}
\end{equation}
for some positive constant $c$. Then for small enough $c$ we have 
$\mcT_{h,I} \subset \mcT_{h,large}$ and we can define the mapping 
$S_h : \mcT_h \setminus \mcT_{h,large} \rightarrow \mcT_{h,large}$. This approach has the 
advantage that fewer elements are mapped resulting in a simpler map $F_h$.
\end{rem}

\subsection{Interpolation}
Here we will show that under the assumption A2, \eqref{eq:approx-Vh}
and using the operator $A_h$ and the space $W_h^E$ constructed in the
previous section we may construct an interpolation operator $\pi_E:L^2(\Omega) \mapsto V_h^E$
with optimal approximation properties. The basic idea is to extend the function 
outside outside of the domain, interpolate the function in $V_h$, restrict to the 
interior elements and then extend using the discrete extension operator.
\begin{itemize}
\item There is a universal extension operator $E:H^s(\Omega) \rightarrow H^s(\IR^d)$ 
such that 
\begin{align}\label{eq:E-stability}
\| E v \|_{H^s(R^d)} \lesssim \| v \|_{H^s(\Omega)}
\end{align}
see \cite{stein70}.

\item Let $\pi_h:H^1(\Omega_h) \rightarrow V_h$ be an interpolation operator of average 
type, see \cite{Cle75} or \cite{ScoZha90}, that 
satisfies the standard element wise estimate 
\begin{equation}\label{eq:interpol-local}
\| v - \pi_h v \|_{H^m(T)} \lesssim h^{s-m} \| v \|_{H^2(\mcT_h(T))}, \qquad 0 \leq m \leq s \leq k+1
\end{equation} 
with $\mcT_h(T)\subset \mcTh$ the neighboring elements of $T$. Composing $\pi_h$ 
with the continuous extension operator $E$ we obtain an interpolation operator 
$\pi_h \circ E : H^1(\Omega) \rightarrow V_h$ and using the 
stability (\ref{eq:E-stability}) of the continuous extension operator we have 
\begin{equation}\label{eq:interpol}
\| E v - \pi_h E v \|_{\mcTh} \lesssim h^{s-m} \| v \|_{H^2(\Omega_h)}
\lesssim 
h^{2-m} \| v \|_{H^2(\Omega)}, \qquad 0 \leq m \leq s \leq k+1
\end{equation}
For simplicity we use the notation $E v   = v$ and $\pi_h v = \pi_h E v$ when appropriate.

\item We define the interpolation operator $\pi_h^E: H^1(\Omega) \rightarrow 
V_h^E$ by 
\begin{equation}
\boxed{
\pi_h^E u = E_h (\pi_h E u)_I
}
\end{equation}
\end{itemize}

%
%
%
%
%

\begin{lem}{\bf (Interpolation Error Estimate).} \label{lem:interpol-full}There is a constant such that 
\begin{equation}\label{eq:interpol-errorest}
\boxed{\|v - \pi_h^E v \|_{H^m(\Omega_h)}  \lesssim h^{k+1-m} \|v\|_{H^{k+1}(\Omega)}, \qquad 0\leq m \leq l+1}
\end{equation}
\end{lem}
\begin{proof}  Follows directly from the following facts, $\pi_h$ is the identity on $V_h$, $\pi_h$ is bounded, the approximation property 
in Lemma \ref{lem:approx}, and the stability (\ref{eq:E-stability}) of the continuous extension operator. 
\end{proof}


\subsection{Some Examples}
\paragraph{Continuous Piecewise Polynomials.} Let $\mcT_h$ be the active mesh 
covering the domain $\Omega$ consisting of simplexes or cubes and consider standard $C^0$ Lagrange elements 
of order $p$. For an element $T\in \mcT_h$ the local finite element space is $P_k(T)$ on simplexes and tensor product 
polynomials $Q_k(T)$ on cubes. Let $\mcX_T$ be the set of nodes associated with the element $T$, and let $\{ v(x) \, : \, \}_{x \in \mcX_T}$ 
be the set of degrees of freedom with corresponding dual basis $\{\varphi^*_x \}_{x \in \mcX_T}$ where $\varphi^*_x(v) = 
v(x)$. The Lagrange basis is defined by 
\begin{equation}
\varphi_x(y) = \delta_{xy}, \qquad x,y \in \mcX_T
\end{equation}


\paragraph{Hermite Splines.} Here we consider the family of tensor product spaces of $C^{(k-1)/2}$ continuous Hermite splines 
of order $k$, where $k$ is an odd number.
\begin{itemize}
\item Let $P_k(I)$ be the space of polynomials of odd order $k$ on the reference interval $I=[0,1]$. The 
dimension of $P_k [0,1]$ is $k+1$, which is even for odd $k$, and the  set of  Hermite degrees of 
freedom, is 
\begin{equation}
\{v^{(l)}(\xi)\, : l=0,1,\dots,(k+1)/2,\, \xi \in \{0,1\} \}
\end{equation}
where $v^{(l)}$ denote the derivative of order $l$ of the function $v$. Here we have $(k+1)/2$ degrees of freedom associated with 
each node $\xi \in \{0,1\}$ and therefore we need the generalized nodes
\begin{equation}
\mcX_I = \{x=(l,\xi) \,:\, l=0,1,\dots,(k+1)/2, \xi \in \{0,1\} \}
\end{equation}
The dual basis is
\begin{equation}
\{ \varphi_x \}_{x \in \mcX_I}
\end{equation}
where for $x=(l, \xi)$ we have $\varphi^*_{(l,\xi)} ( v ) = v^(l)(\xi)$. Finally, the Lagrange basis $\{\varphi_x\}_{x \in \mcX_I}$ 
is defined by the equations $\varphi^*_x(\varphi_y) = \delta_{xy}$, $x,y \in \mcX_I$, which means that for 
$(l,\xi),  (\widetilde{l},\widetilde{\xi})\in \mcX_I$, 
\begin{align}
\varphi^{(\widetilde{l})}_{(l,\xi)}(\widetilde{\xi}) =
\begin{cases}
1 & \text{$l = \widetilde{l}$ and $\xi = \widetilde{\xi}$}
\\
0 & \text{otherwise}
\end{cases}
\end{align}

\item Let $\widetilde{\mcT}_h$, $h\in (0,h_0]$, be a family of 
partitions of $\IR^d$ into cubes with side $h$.  Let $\widetilde{V}_h$ be 
the space consisting of tensor products of odd order Hermite 
splines on $\widetilde{\mcT}_h$. 

\item 
Let $\mcT_h = \{T \in \widetilde{\mcT}_h : T \cap \Omega \neq \emptyset\}$ 
be the active mesh. Let $V_h$ be the restriction of $\widetilde{V}_h$ to  $\mcT_h$.

%

\end{itemize}

\paragraph{Nonconforming Elements.} Our framework applies to nodal nonconforming piecewise polynomial elements, for 
instance, the Morley elements and the Crouzeix-Raviart elements. It is however important to note that the error analysis of these 
elements rely on the orthogonality properties of the discontinuities at the faces, which in general does not hold for faces that are 
cut since then only part of the integral is present in the form. Using 
a discontinuous Galerkin formulation on all faces that intersect the boundary we obtain a stable method with optimal order convergence.
 Let us consider the Crouzeix-Raviart elements for simplicity. The nodes $\mcX_T$ associated with the simplex $T$ is the midpoints of 
 the faces and the degrees of freedom are the function values in the midpoints. Then the average of the jump in the finite element functions 
 are  zero for all faces residing in the interior of $\Omega$, while for faces that cuts the boundary this is not the case. Therefore on all faces 
 intersecting the boundary we add the standard symmetric interior penalty terms leading to a method with optimal order convergence.

\section{Abstract Framework for CutFEM using Extended FE Spaces
  and Nitsche's Method}\label{sec:abstract}
In this section we apply our framework to an abstract Nitsche method which can be used to analyse several relevant situations including 
boundary and interface problems of different order. 

Consider approximating an abstract boundary value problem: find $u \in V_{bc} \subset V$ such that 
\begin{align}\label{eq:abstract-mod-prob}
a_\Omega(u,v) = l(v)\qquad \forall v \in V_{bc,0}
\end{align}
where the boundary conditions are strongly enforced in $V_{bc}$ and $V_{bc,0}$ is the corresponding space with homogeneous 
boundary conditions. We assume that $a_\Omega$ is continuous and coercive, and that $l$ is continuous. Then it follows from the Lax-Milgram lemma that there is a unique solution to (\ref{eq:abstract-mod-prob}).

Next consider an abstract Nitsche type approximation of (\ref{eq:abstract-mod-prob}) with weak enforcement of the boundary 
conditions : find $u_h \in V_h^E$ such that 
\begin{align}\label{eq:abstract-nit}
a_h(u_h,v) = l_h(v)\qquad \forall v \in V_h^E
\end{align}
The form $a_h$ is defined by
\begin{align}
a_h(v,w) = a_\Omega(v,w) - a_{\partial \Omega}(v,w) - a_{\partial \Omega}(w,v) + \beta b(v,w)
\end{align} 
and  $l_h$ is defined by
\begin{equation}
l_h(v) = l(v) - a_{\partial \Omega}(v,u) + \beta b(v,u)
\end{equation}
The rationale for the Nitsche formulation is to extend the bilinear
form $a$ to $a_h$ in such a way that the solution to
\eqref{eq:abstract-mod-prob} also is a solution to
\eqref{eq:abstract-nit}. In particular we require
\begin{equation}
a_\Omega(u,v) - a_{\partial \Omega}(u,v)= l(v) \qquad \forall v \in V_h^E
\end{equation}
However, since the test space in \eqref{eq:abstract-nit} no longer satisfies boundary conditions this
may require some additional regularity of $u$ so that the form $ a_{\partial
  \Omega}(v,w)$ is well defined for $v=u$, we formally denote the
space of functions with the required additional regularity by $\widetilde{V}$ and note that $\widetilde{V} \subset V$. 
Observe that we do not require the problem \eqref{eq:abstract-nit} to be well posed in the sense 
that the form $a_h$ is coercive on the continuous level but it should be well defined.

We assume that the following properties hold.
\begin{description}
\item[B1] There is a norm $\tn \cdot \tn_{\Omega}$ on $\widetilde{V} + V_h^E$ such that the form $a_\Omega$ 
is continuous
\begin{equation}
a_\Omega(v,w)\lesssim \tn v \tn_{\Omega} \tn w \tn_{\Omega} \qquad v,w \in \widetilde{V} +V_h^E
\end{equation}
 and coercive
\begin{equation}\label{eq:B1-coer}
\tn v \tn^2_{\Omega} \lesssim a_\Omega(v,v)\qquad v \in V_h^E
\end{equation}

\item[B2] The form $b$ induces a seminorm $\|\cdot \|_b$ on $\widetilde{V} + V_h^E$,
and there is a seminorm $\tn \cdot \tn_{\partial \Omega}$ on $\widetilde{V} + V_h^E$ such that 
\begin{equation}\label{eq:abs-bnd-cont}
|a_{\partial \Omega}(v,w)| \lesssim \tn v \tn_{\partial \Omega} \| w
\|_b \qquad v,w \in \widetilde{V} + V_h^E
\end{equation}

\item[B3] The seminorm $\tn \cdot \tn_{\partial \Omega}$ satisfies the
  inverse estimate 
\begin{equation}\label{eq:abs-inverse}
\tn v \tn_{\partial \Omega} \lesssim \tn v \tn_{\Omega}\qquad v \in V_h^E
\end{equation}
and as a consequence of (\ref{eq:abs-bnd-cont}) it follows that
\begin{equation}\label{eq:abs-bnd-disc}
|a_{\partial \Omega}(v,w)| \lesssim \tn v \tn_{\Omega} \| w \|_b \qquad v,w \in V_h^E
\end{equation}

\item[B4] The functional $l_h$ is continuous on $V_h^E$
\begin{equation}
|l_h(v)| \lesssim \tn v \tn_{h}\qquad v \in V_h^E
\end{equation}
where the energy norm is defined by
\begin{equation}
\tn v \tn^2_h = \tn v \tn^2_\Omega + \tn v \tn^2_{\partial \Omega} + \| v\|^2_b
\end{equation}

\item[B5]  The method  (\ref{eq:abstract-nit}) is consistent in the
  sense that for $u \in \widetilde{V}$, solution to (\ref{eq:abstract-mod-prob}) 
satisfies (\ref{eq:abstract-nit}), 
\begin{equation}\label{eq:consistency}
a_h(u,v) = l_h(v) \qquad \forall v \in \widetilde{V} + V_h^E
\end{equation}
\end{description}
\begin{rem} Note that the coercivity (\ref{eq:B1-coer}) typically holds for a larger space than $V_h^E$, but since the coercivity 
for the Nitsche method, which we establish in (\ref{eq:coerciv}) below, only holds on $V_h^E$ it is enough to assume coercivity of 
$a_\Omega$ on $V_h^E$. 

\end{rem}
\begin{rem}
The norms $\tn \cdot \tn_h$, $\tn \cdot \tn_{\partial \Omega}$, and $\| \cdot \|_b$ are in general mesh dependent norms, and we will 
specify them precisely in the forthcoming examples. In fact in assumptions B1-B5 the index $h$ is only used to indicate the discrete space and 
the discrete forms.
\end{rem}
\begin{rem} The key property for cut finite element methods is the inverse inequality (\ref{eq:abs-inverse}) in B3, which in general does 
not hold without some modification of the method or finite element space. For instance, adding some type of stabilization such as least squares 
control over the jumps in derivatives across faces or, as in this paper, using an extended finite element space.
\end{rem}

\subsection{Properties of the Abstract Method}
Starting from the assumptions B1-B5 we derive the key properties of the abstract Nitsche method.

\begin{lem} If B1-B3 hold, then the form $a_h$ is continuous 
\begin{equation}\label{eq:continuity}
a_h(v,w)\lesssim \tn v \tn_{h} \tn w \tn_{h} \qquad v,w \in \widetilde{V}+V_h^E
\end{equation}
and for $\beta$ large enough coercive 
\begin{equation}\label{eq:coerciv}
\tn v \tn^2_{h} \lesssim a_h(v,v)\qquad v \in V_h^E
\end{equation}
\end{lem}
\begin{proof} Continuity follows directly from B1-B2, 
\begin{align}
a_h(v,w) &=  a_\Omega(v,w) - a_{\partial \Omega}(v,w) - a_{\partial \Omega}(w,v) + \beta b(v,w)
\\
&\leq \tn v \tn_h \tn w \tn_\Omega + \tn v \tn_{\partial \Omega} \| w \|_b + \tn w \tn_{\partial \Omega} \| v \|_b + \beta \| v \|_b \| w \|_b
\\
&\leq\max(1,\beta) (\tn v \tn^2_h + \tn v \tn^2_{\partial \Omega} + \| v \|^2_b)^{1/2} (  \tn w \tn^2_\Omega +  \tn w \tn^2_{\partial \Omega}  + \| w \|^2_b )^{1/2}
\\
&\lesssim 
\tn v \tn_h \tn w \tn_h
\end{align}
Coercivity follows using B1-B3 and in particular (\ref{eq:abs-bnd-disc}), 
\begin{align}
\tn v \tn_h^2 &=a_h(v, v) 
\\
&= \tn v \tn_\Omega^2 - 2 a_{\partial \Omega} (v, v) + \beta \| v \|^2_b
\\
&\geq  \tn v \tn_\Omega^2 - 2 C \tn v \tn_\Omega \| v \|_b + \beta \| v \|^2_b
\\
&\geq  \tn v \tn_\Omega^2 -  \delta C^2 \tn v \tn^2_\Omega + \delta^{-1} \| v \|_b + \beta \| v \|^2_b
\\
&\geq (1 - C^2 \delta )  \tn v \tn_\Omega^2 + (\beta - \delta^{-1} )\| v \|^2_b
\end{align}
Taking $\delta$ small enough, and $\beta$ large enough we obtain 
\begin{equation}
 \tn v \tn_\Omega^2 +\| v \|^2_b \lesssim a_h(v,v) \qquad v \in V_h^E
\end{equation}
Finally using (\ref{eq:abs-inverse}) the coercivity follows.
\end{proof}

\begin{thm} \label{thm:best_approx}  If B1-B5 hold there exists a unique solution to
  (\ref{eq:abstract-nit}) and the following best approximation estimate holds
\begin{equation}\label{eq:abs-best-approx}
\boxed{\tn u - u_h \tn_h \lesssim \tn u - v \tn_h \qquad \forall v \in V_h^E}
\end{equation}
\end{thm}
\begin{proof} Since $a_h$ is coercive and continuous on $V_h^E$ and according to B4 the functional $l_h$ is continuous on $V_h^E$ it 
follows from the Lax-Milgram lemma that there is a unique solution $u_h \in V_h^E$ to (\ref{eq:abstract-nit}).

To prove the error estimate (\ref{eq:abs-best-approx}) we add and subtract $v \in V_h^E$, and using the triangle inequality we then have 
\begin{align}\label{eq:abs-lem-err-a}
\tn u - u_h \tn_h \leq \tn u - v \tn_h + \tn v - u_h \tn_h 
\end{align}
for the second term we use the fact that $v-u_h \in V_h^E$ and apply the coercivity, then we add and subtract the exact solution $u$, 
employ the consistency, and finally use the continuity to conclude that  
\begin{align}
\tn v - u_h \tn_h^2 &\lesssim a_h(v - u_h,v-u_h) 
\\
&= a_h(v - u,v-u_h) + a_h(u - u_h,v-u_h) 
\\
&= a_h(v - u,v-u_h) + \underbrace{a_h(u,v-u_h) - l_h(v-u_h)}_{=0}
\\
&= a_h(v - u,v-u_h)
\\
&\lesssim  \tn v - u \tn_h \tn_h \tn v - u_h \tn_h  
\end{align}
Thus we have 
\begin{align}
\tn v - u_h \tn_h \lesssim  \tn v - u \tn_h
\end{align}
which combined with (\ref{eq:abs-lem-err-a}) completes the proof of (\ref{eq:abs-best-approx}).
\end{proof}

Assuming that we have a family of finite element spaces, with mesh parameter $h\in (0,h_0]$, which 
satisfies the approximation property 
\begin{equation}\label{eq:abs_approx}
\inf_{w \in V_h^E} \tn v - w \tn_h \lesssim h^{k-l} \| u \|_{H^k(\Omega)}
\end{equation}
where $k$ is the approximation order of the finite element space, we obtain the following error estimate 
for an elliptic operator of order $2l$, 
\begin{equation}
\tn u - u_h \tn_h \lesssim  h^{k-l} \| u \|_{H^k(\Omega)}
\end{equation}

Error estimates in weaker norms can be obtained if an additional regularity
assumption holds. We assume that the following elliptic shift estimate is
satisfied by the solution $u \in V_{bc,0}$ to \eqref{eq:abstract-mod-prob}, 
\begin{equation}\label{eq:abs_elliptic}
|u|_{H^{2l-s}(\Omega)} \lesssim \|f\|_{H^{-s}(\Omega)}, \qquad 0\leq s \leq l
\end{equation}
where $2l$ is the order of the operator.

\begin{thm}\label{eq:L2_abs}  Assuming that assumptions B1-B5, the approximation property  
\eqref{eq:abs_approx} with $k = 2l$, and the elliptic shift estimate (\ref{eq:abs_elliptic}) hold. Then
\begin{equation}\label{eq:abstract_L2}
\boxed{
\|u - u_h\|_{H^s(\Omega)} \lesssim h^{l-s} \tn u - v \tn_h \qquad \forall v \in V_{h,E}, \qquad 0\leq s \leq l
}
\end{equation}
\end{thm}
\begin{proof}
We will argue by duality and therefore let $\phi \in V_{bc,0}$ solve the dual problem
\begin{equation}
a_\Omega(v, \phi) = l_\psi(v) \qquad \forall v \in V_{bc,0}
\end{equation}
where $l_\psi \in V_{bc,0}^*$ takes the form
\begin{equation}
l_\psi(v) =\langle  \psi, v \rangle_s 
\end{equation}
where $\langle \cdot , \cdot \rangle_s: H^{-s}(\Omega) \times  H^{s}(\Omega) \rightarrow \IR$ is the duality pairing.

By symmetry of $a_h$ and consistency (\ref{eq:consistency}) in assumption B5 it follows that $\phi$ satisfies the 
adjoint consistency 
\begin{align}\label{eq:consistency-adjoint}
a_h(v,\phi ) = l_\psi(v) \qquad \forall v \in \widetilde{V} + V_h^E
\end{align}
Setting $v = u - u_h$ in (\ref{eq:consistency-adjoint})  we get 
\begin{align}
\langle u- u_h,\psi\rangle_s 
&=a_h(u - u_h, \phi ) 
\\
&=a_h(u - u_h, \phi - w ) 
\\
&\lesssim \tn u - u_h \tn_h \tn \phi - \phi_h \tn_h
\\
&\lesssim \tn u - u_h \tn_h h^{l-s} | \phi |_{H^{2l-s}(\Omega)}
\\
&\lesssim h^{l-s}  \tn u - u_h \tn_h \| \psi \|_{H^s(\Omega)}
\end{align}
where we used the consistency (\ref{eq:consistency}) to subtract $w \in V_h^E$, the continuity (\ref{eq:continuity}) of $a_h$, 
the approximation property (\ref{eq:abs_approx}), and the elliptic regularity (\ref{eq:abs_elliptic}). We therefore arrive at 
\begin{equation}
\| u - u_h \|_{H^s(\Omega)} = \sup_{\psi \in H^{-s}(\Omega)} \frac{\langle u- u_h,\psi\rangle_s}{\| \psi \|_{H^{-s}(\Omega)}} 
\lesssim  h^{l-s}  \tn u - u_h \tn_h 
\end{equation}
which completes the proof.

\end{proof}

\subsection{Time Dependent Problems}\label{sec:time}
The power of the abstract framework established above is that once
stability and optimal accuracy has been established for the
Ritz-projection associated to the (time constant coefficient) elliptic model problem
\eqref{eq:abstract-mod-prob} we can immediately extend the results to cut finite
element methods for the associated time dependent problems. To
illustrate this we will consider the abstract parabolic problem
subject to the elliptic operator $a$ of \eqref{eq:abstract-mod-prob}. An identical
argument can be developed for the second order hyperbolic problem, for
details on this we refer to \cite{BHL20b}. In this reference it is also
shown that the discrete extension makes it
possible to lump the mass matrix for explicit time-stepping.  For simplicity 
we consider only semi-discretization in space, however the arguments extend in a
straightforward way to the fully discrete case using any state of the
art time discretization for parabolic problems \cite{Thom06}. 

First we introduce the Ritz projection, $R_h: \widetilde{V} \mapsto V_h^E$, where we recall 
that $\widetilde{V}$ is $V$ with some more smoothness to guarantee that $a_h$ is defined on 
$\widetilde{V}$, defined by
\begin{equation}\label{eq:Ritz}
a_{h}(R_h v,w) = a_h(v,w) \qquad \forall w \in V_h^E
\end{equation}
Differentiating \eqref{eq:Ritz} in time we see that $\partial^i R_h v =
R_h \partial_t v$, since $R_h$ is independent of time.
Assuming that assumptions B1-B5 hold, it follows from Theorem \ref{thm:best_approx} and Theorem
\ref{eq:L2_abs} that for $i \in \{0,1\}$,
\begin{equation}
\|\partial_t^i  (v - R_h v)\|_\Omega + h^l \tn \partial_t^i (v - R_h v_h) \tn_h\lesssim h^l \tn \partial_t^i (v - w)\tn_h \qquad \forall w
\in V_h^E
\end{equation}

Let $I= (0,T)$ be a time interval and $Q:= \Omega \times I$ the space time domain and consider the problem, 
find $u \in V^{bc}_Q := L^2(0,T;V_{bc})$, $u(\cdot,0) =
u_0 \in \widetilde{V}_{bc}$ such that
\begin{equation}\label{eq:parabolic}
(\partial_t u, v)_Q+ a_Q(u,v) = l_Q(v) \qquad \forall  v \in V_Q^0
\end{equation}
where $V_Q^0 :=  L^2(0,T;V_{bc,0})$, 
\begin{equation}
(u,v)_Q = \int_0^T (\partial_t u, v)_\Omega 
\end{equation}
and
\begin{equation}
l_Q =  \int_0^T l(v) 
\end{equation}
with $l$ a given linear functional that may depend on time.
For all $l_Q \in V_Q'$ the problem \eqref{eq:parabolic} admits a
unique solution \cite[Theorem 4.1 and Remark 4.3]{LM72}. 

We propose the following CutFEM discretization of the problem
\eqref{eq:parabolic}. Find $u_h:[0,T] \rightarrow V_h^E$ such that for all
$t \in (0,T)$ there holds
\begin{equation}\label{eq:parabolic_FEM}
(\partial_t u_h, v)_\Omega+ a_h(u_h,v) = l(v)\qquad \forall v \in V_h^E
\end{equation}
The equation \eqref{eq:parabolic_FEM} can now be discretized in time,
for instance
by replacing $\partial_t$ with any suitable finite difference method
such as backward differentiation or Crank-Nicolson and evaluate $u_h$
at a suitable point in time in $a_h$. For the backward Euler method
the linear system associated to one time step takes the well-known form:
find $u_h^{n+1} \in V_h^E$ such that
\begin{equation}
\tau^{-1} (u_h^{n+1},v)_\Omega +a_h(u_h^{n+1},v) = l_{n+1}(v) +
\tau^{-1} (u_h^{n+1},v)_\Omega \qquad \forall v \in V_h^E
\end{equation}
We see that this linear system is stable indepently of the the
mesh/interface intersection thanks to the stability of the extended
space, see Lemma \ref{lem:stab_E}.

The following error estimate holds for the semi-discretized problem
\eqref{eq:parabolic_FEM}.
\begin{thm}
Let $u_h$ be the solution of \eqref{eq:parabolic_FEM} and $u$ the
solution of \eqref{eq:parabolic} then there holds
\begin{equation}
\sup_{t \in (0,T)} \|u(t)- u_h(t)\|_\Omega \lesssim \|u(0) -
u_h(0)\|_\Omega + \int_0^T h^{l}
\inf_{v_h \in V_h^E} \tn \partial_t v - v_h \tn_h 
\end{equation}
and 
\begin{align}
\int_0^T \tn u- u_h \tn^2_h   &\lesssim \|u(0) -
u_h(0)\|_\Omega^2+\int_0^T \inf_{v_h \in V_h^E} \tn v - v_h \tn^2_h 
\\
&\qquad 
+ \left(\int_0^T h^{l} \inf_{v_h \in V_h^E} \tn \partial_t v - v_h \tn_h \right)^2
\end{align}
\end{thm}
\begin{proof}
The proof uses standard arguments for the parabolic problem together with
the CutFEM toolbox for elliptic problems developed above.
First we decompose the error as
\begin{equation}
u-u_h = \underbrace{u - R_h u}_{e_R} + \underbrace{R_h u - u_h}_{e_h}
\end{equation}
Since the estimate for $e_R$ is immediate using \eqref{eq:Ritz} we
only need to prove the bounds for the discrete error $e_h$. Using the
formulation \eqref{eq:parabolic_FEM} and the coercivity of the form
$a_h$, for $\beta$ sufficiently large, \eqref{eq:coerciv}, there exists a constant $\alpha>0$ 
such that for $s \in (0,T)$
\begin{equation}
\|e_h(s)\|_\Omega^2 + \alpha \int_0^s \tn e_h \tn_h^2 
\leq  \|e_h(0)\|_\Omega^2  + \int_0^s (\partial_t e_h, e_h)_\Omega  + \int_0^s a_h(e_h,e_h)
\end{equation}
For the right hand side we see that using \eqref{eq:parabolic} the following Galerkin orthogonality 
holds
\begin{equation}
\int_0^s (\partial_t (u - u_h), e_h)_\Omega + \int_0^s a_h(u - u_h,e_h)  = 0
\end{equation}
and by the definition of the Ritz projection $a_h(R_h u - u,e_h) =
0$ which imply
\begin{equation}
\|e_h(s)\|_\Omega^2+ \int_0^s a_h(e_h,e_h)   =  \|e_h(0)\|_\Omega^2-\int_0^s
(\partial_t e_R, e_h)_\Omega
\end{equation}
Taking the sup over $s \in (0,T)$ we then obtain
\begin{equation}
\sup_{s \in (0,T)} \|e_h(s)\|_\Omega^2 \leq \|e_h(0)\|_\Omega^2 +
\sup_{s \in (0,T)} \|e_h(s)\|_\Omega \int_{0}^T \| \partial_t e_R
\|_\Omega 
\end{equation}
and therefore 
\begin{equation}
\sup_{s \in (0,T)} \|e_h(s)\|_\Omega^2 \lesssim  \|e_h(0)\|^2_\Omega +
\Big( \int_{0}^T \| \partial_t e_R \|_\Omega \Big)^2
\end{equation}

Applying \eqref{eq:Ritz} with $i=1$ we see that
\begin{equation}
\sup_{s \in (0,T)} \|e_h(s)\|_\Omega \leq C \int_{0}^T h^{l}
\inf_{w \in V_h^E} \tn \partial_t (v - w) \tn_h 
\end{equation}
The triple norm bound follows by observing that
\begin{equation}
 \int_0^T a_h(e_h,e_h)  \leq \|e_h(0)\|_\Omega^2+\sup_{s \in (0,T)} \|e_h(s)\|_\Omega^2
 + \left(\int_{0}^T \| \partial_t e_R
\|_\Omega\right)^2
\end{equation}
\end{proof}

\section{Applications}\label{sec:appli}
To show the flexibility of the above framework we will below consider
some different partial differential equations that enter the
framework. In principle the arguments of the abstract framework can be
applied to elliptic operators of any order $2l$, $l=1,2,3...$. However
for the ske of conciseness we only discuss the cases up to $l=3$.

\subsection{Second Order Boundary Value Problems}

\paragraph{The Model Problem.}
Consider the second order boundary value problem
\begin{equation}\label{eq:poisson}
-\Delta u = f \qquad \text{in $\Omega$}, \qquad u = g \qquad \text{on $\partial\Omega$}
\end{equation}
For smooth boundary there is a unique solution to this problem and we have the elliptic regularity 
\begin{equation}
\| u \|_{H^{s+2}(\Omega)} \lesssim \| f \|_{H^s(\Omega)} + \| g \|_{H^{s+3/2}(\partial \Omega)}
\end{equation}
\paragraph{The Finite Element Method.}
The standard Nitsche method takes the form
\begin{align}\label{eq:sec-a}
a_h(u_h,v) = l_h(v)
\end{align}
where 
\begin{align}
a_h(v,w) &=(\nabla v, \nabla w)_\Omega - (\nabla_n v, w)_{\partial \Omega} - (\nabla_n w, v)_{\partial \Omega} 
+ \beta h^{-1} (v,w)_{\partial \Omega}
\\
l_h(v) &= (f,v)_{\Omega} - (g, \nabla_n v)_{\partial \Omega}  + \beta h^{-1} (g,v)_{\partial \Omega}
\end{align}
Setting 
\begin{align}
a_\Omega (v,w) &= (\nabla v, \nabla w)_\Omega 
\\
a_{\partial \Omega} (v,w) &= (\nabla_n v, w)_{\partial \Omega} 
\\
b(v,w) &= h^{-1} (v,w)_{\partial \Omega}
\end{align}
and
\begin{align}
\tn v \tn^2_\Omega &= \| \nabla v \|^2_\Omega
\\
\tn v \tn^2_{\partial \Omega} &= h \| \nabla_n v \|^2_{\partial \Omega}
\\
\| v \|^2_b &=  h^{-1}\| v \|^2_{\partial \Omega}
\end{align}
we translate the problem (\ref{eq:sec-a}) into the abstract framework and it remains to verify
assumptions B1-B5. Here {B1} and {B2} follows directly from the Cauchy-Schwarz inequality. In 
{B3} the key estimate (\ref{eq:abs-inverse})  takes the form 
\begin{equation}
h \| \nabla_n v \|^2_{\partial \Omega} \lesssim \| \nabla v \|^2_\Omega \qquad v \in V_h^E
\end{equation}
Using the inverse  inequality, see \cite{HWX17},  
\begin{equation}
h \| w \|^2_{T\cap \partial \Omega} \lesssim \| w \|^2_T  \qquad w \in \mathbb{P}_k(T)
\end{equation}
applied to $w = \nabla v$ we get 
\begin{align}
h \| \nabla_n v \|^2_{\partial \Omega} \lesssim h \| \nabla v \|^2_{\partial \Omega} 
\lesssim 
\| \nabla v \|^2_{\mcT_h (\partial \Omega)} 
\lesssim
\| \nabla v \|^2_{\Omega_h} 
\lesssim 
\| \nabla v \|^2_\Omega
\end{align}
Here we finally used the stability (\ref{eq:E-stability}) of the extension operator. To verify {B4} we use the 
Cauchy-Schwarz inequality, 
\begin{align}
l_h(v) =& (f,v)_\Omega- (g, \nabla_n v)_{\partial \Omega}  + \beta h^{-1} (g,v)_{\partial \Omega}
\\
&\leq  \| f \|_\Omega \|v\|_\Omega + h^{-1/2}\|g\|_{\partial \Omega}  h^{1/2}\|\nabla_n v\|_{\partial \Omega}  
+ \beta h^{-1/2} \|g\|_{\partial \Omega} h^{-1/2}\| v\|_{\partial \Omega}
\\
&\leq \max(1,\beta)  (\| f \|^2_\Omega  + h^{-1}\|g\|^2_{\partial \Omega} )^{1/2}  (     \|v\|^2_\Omega  +  h \|\nabla_n v\|^2_{\partial \Omega}  
+ h^{-1}\| v\|^2_{\partial \Omega} )^{1/2}
\\
&\lesssim  (\| f \|^2_\Omega  + h^{-1}\|g\|^2_{\partial \Omega} )^{1/2} \tn v \tn_h
\\
&\lesssim_{h^{-1/2}} \tn v \tn_h
\end{align}
which for fixed $h$ proves the desired continuity. Note that we only use the continuity of $l_h$ to conclude that there is a unique 
solution to the discrete problem by application of the Lax-Milgram lemma, and therefore we apply the stability for fixed mesh parameters. 
Finally, the consistency {B5} follows directly from an application of Green's formula.

\paragraph{Error Estimate.}
To turn the abstract error estimate (\ref{eq:abs-best-approx}) into a quantitative bound we use the interpolation theory 
for $V_h^E$ to show that 
\begin{equation}
\tn u - u_h \tn_h \lesssim 
\tn u - \pi_h u \tn_h \lesssim h^{k-1} \| u \|_{H^{k}(\Omega)} 
\end{equation}
which for instance holds $C^0$ Lagrange elements of order $k$.

\subsection{Second Order Interface Problems}

\paragraph{The Model Problem.}
Let $\Omega \subset  \IR^d$ be a polygonal domain.  Let  $\Omega_1 \subset \Omega \setminus U_\delta(\partial \Omega)$, where 
$U_\delta(\partial \Omega) = \{ x \in \IR^d \, |\, \text{dist}(x,\partial \Omega) < \delta\}$, be a subset 
with smooth boundary $\partial \Omega_1$, which also forms the interface $\Gamma$,  and let $\Omega_2 = \Omega \setminus \Omega_1$. 
Consider the interface problem 
\begin{alignat}{3}\label{eq:interface-a}
-\nabla \cdot A_i \nabla u_i &=f_i &\qquad &\text{in $\Omega_i$}
\\ \label{eq:interface-b}
[u_i] & = 0 &\qquad &\text{on $\Gamma$}
\\ \label{eq:interface-c}
[n \cdot A_i \nabla u_i]&=0 &\qquad &\text{on $\Gamma$}
\\\label{eq:interface-d}
u&=0 &\qquad &\text{on $\partial \Omega$}
\end{alignat}
where $A_i$ are constant positive definite matrices. Testing with $v \in H^1_0(\Omega)$ and integrating by parts and using the 
interface condition we obtain the weak form 
\begin{align}
\sum_{i=1}^2 (f_i,v)_{\Omega_i} &= \sum_{i=1}^2 -(\nabla \cdot A_i \nabla u_i,v)_{\Omega_i} 
\\
&\qquad =   \sum_{i=1}^2 (A_i \nabla u_i, \nabla v)_{\Omega_i} - ([n \cdot A_i \nabla u_i], v)_\Gamma
 =\sum_{i=1}^2 (A_i \nabla u_i, \nabla v)_{\Omega_i}
\end{align}
and we note that the form on the right hand side is coercive and continuous on $H^1_0(\Omega)$ and we can conclude using Lax-Milgram 
that there is an exact solution in $H^1_0(\Omega)$. 

\paragraph{The Finite Element Method.}

Let $V_{h_i}^E$ be finite element spaces on $\Omega_i$ that extends over $\Gamma$. For simplicity we assume that the 
homogeneous boundary conditions on the external boundary $\partial \Omega$ are strongly enforced in $V_{h,2}^E$ using a 
matching mesh at $\partial \Omega$. The finite element method takes the form: find $u_h = (u_{h,1},u_{h,2}) \in V_{h,1}^E \oplus V_{h,2}^E = V_h^E$, such that 
\begin{equation}\label{eq:interface-method}
a_h(u_h, v) = l_h(v) \qquad v \in V_h^E
\end{equation}
where the forms are 
\begin{align}
a_h(v,w) &= \sum_{i=1}^2 (A_i \nabla v, \nabla w)_{\Omega_i}  -  (n_i \cdot A_i \nabla v_i, w_i - \langle w \rangle)_{\partial \Omega_i }
\\
&\qquad
- (n_i \cdot A_i \nabla w_i, v_i - \langle v \rangle)_{\partial \Omega_i }
+ \beta h^{-1} \| n_i \|_{A_i} ( v_i - \langle v \rangle , w - \langle w_i \rangle )_{\partial \Omega_i}
\\
l_h(v) &= \sum_{i=1}^2 (f_i, v_i)_\Omega
\end{align}
with $\| n_i \|^2_{A_i} = n_i \cdot A_i \cdot n_i$ and $\langle \cdot \rangle$ is an convex combination average at 
the interface $\Gamma$ defined by 
\begin{equation}\label{eq:weights}
\langle \cdot \rangle : V_h^E \ni (v_1,v_2) \mapsto \sum_{i=1}^2 \kappa_i v_i \in \sum_{i=1}^2  (V_{h,i}^E)|_\Gamma 
\end{equation}
with $\kappa_i>0$ and $\kappa_1 + \kappa_2 = 1$.
\paragraph{The Abstract Setting.}
The method is transferred into the abstract framework by working in the finite element space $V_h^E = V_{h,1}^E \oplus V_{h,2}^E$ 
and defining the forms 
\begin{align}
a_\Omega (v,w) &= \sum_{i=1}^2 (A_i \nabla v, \nabla w)_{\Omega_i} 
\\
a_{\partial \Omega} (v,w) &= \sum_{i=1}^2 (n_i \cdot A_i \nabla v_i, w_i - \langle w \rangle)_{\partial \Omega_i }
\\
b(v,w) &= \sum_{i=1}^2  \beta_i h^{-1} \| n_i \|_{A_i} ( v_i - \langle v \rangle , w - \langle w_i \rangle )_{\partial \Omega_i}
\end{align}
and norms
\begin{align}
\tn v \tn^2_\Omega &=  \sum_{i=1}^2 (A_i \nabla v_i, \nabla v_i)_{\Omega_i}
\\
\tn v \tn^2_{\partial \Omega} &=  \sum_{i=1}^2 h \| n_i \|_{A_i}^{-1} \| n_i \cdot A_i \nabla v_i \|^2_{\partial \Omega_i}
\\
\| v \|^2_b &=  \sum_{i=1}^2  \beta_i h^{-1} \| n_i \|_{A_i} \|v_i - \langle v \rangle\|^2_{\partial \Omega_i}
\end{align}
Next we verify the assumptions. 

\begin{rem} Our formulation of the finite element method is equivalent to standard Nitsche formulations for the 
interface problem but it has a simpler structure only involving the average of the solution at the interface avoiding 
introduction of the average of the flux and jump which are quantities
with signs depending on the order of the subdomains. This also
connects in a natural way to hybridised methods simply by replacing
the average $\langle v \rangle$ by a trace variable.
Note that we get two subdomain Nitsche formulations where the Dirichlet data is precisely the average $\langle u_h \rangle$ 
which leads to a simple decoupled structure. To verify that the method is indeed equivalent to a standard Nitsche formulation 
we observe that $v_1 - \langle v \rangle = (1- \kappa_1) v_1 -  \kappa_2 v_2 = \kappa_2 (v_1 - v_2 )$ and similarly 
$v_2 - \langle v \rangle = \kappa_1 (v_2 - v_1 )$, which gives the identity 
\begin{align}
&\sum_{i=1}^2  (n_i \cdot A_i \nabla v_i, w_i - \langle w \rangle)_{\partial \Omega_i }
\\
&\qquad =
\kappa_1^* (n_1 \cdot A_1 \nabla v_1, w_1 - w_2 )_{\partial \Omega_i }
+
\kappa_2^* (n_2 \cdot A_2 \nabla v_2, w_2 - w_1 )_{\partial \Omega_i }
\\
&\qquad =
\kappa_1^* (n_1 \cdot A_1 \nabla v_1, w_1 - w_2 )_{\partial \Omega_i }
+
\kappa_2^* (n_1 \cdot A_2 \nabla v_2, w_1 - w_2 )_{\partial \Omega_i }
\\
&\qquad =
\langle n \cdot A \nabla v \rangle_*, [w] )_{\Gamma }
\end{align}  
where $\kappa_i^* = 1 - \kappa_i$ are the dual weights and we defined the average of the flux
\begin{align}
\langle n \cdot A \nabla v \rangle_* =\kappa_1^* n_1 \cdot A_1 \nabla v_1 + \kappa_2^* n_1 \cdot A_2 \nabla v_2
\end{align} 
and the jump 
\begin{align}
[v] = v_1 - v_2
\end{align}
In a similar way we have
\begin{align}
b(v,w) &= \sum_{i=1}^2  \beta_i h^{-1} \| n_i \|_{A_i} ( v_i - \langle v \rangle , w - \langle w_i \rangle )_{\partial \Omega_i}
\\
&= \left( \sum_{i=1}^2  \beta_i h^{-1} \| n_i \|_{A_i} (\kappa_i^*)^2 \right) ( [ v ] , [w] )_{\partial \Omega_i}
\end{align}
Thus our formulation is indeed equivalent to a standard Nitsche formulation. 
\end{rem}

\paragraph{Verification of Assumptions.} {B1} follows directly from the fact that matrices $A_1$ and $A_2$ are constant and positive 
definite. For {B2} we note that using the Cauchy-Schwarz inequality we directly obtain the estimate 
\begin{align}
a_{\partial \Omega} (v,w) &= \sum_{i=1}^2 (n_i \cdot A_i \nabla v_i, w_i - \langle w \rangle)_{\partial \Omega_i } 
\\
&\leq  \sum_{i=1}^2 h^{1/2} \| n_i \|_{A_i}^{-1/2} \|n_i \cdot A_i \nabla v_i\|_{\partial \Omega_i}   h^{-1/2} \| n_i \|_{A_i}^{1/2}  \|w_i - \langle w \rangle\|_{\partial \Omega_i } 
\\
&\leq  \left( \sum_{i=1}^2 h \| n_i \|_{A_i}^{-1} \|n_i \cdot A_i \nabla v_i\|^2_{\partial \Omega_i}  \right)^{1/2} 
\left( \sum_{i=1}^2 h^{-1} \| n_i \|_{A_i}  \|w_i - \langle w \rangle\|^2_{\partial \Omega_i } \right)^{1/2}
\\
&\leq \tn v \tn_{\partial \Omega} \| w \|_b
\end{align}
For {B3} we proceed with standard estimates, use the fact that $\| n_i \|_{A_i} = \| n_i\|^2_{A_i}$,  followed by an inverse 
inequality to pass from the boundary to the set of elements intersecting the boundary 
\begin{align}
&\tn v \tn^2_{\partial \Omega} = \sum_{i=1}^2 h \| n_i \|_{A_i}^{-1} \|n_i \cdot A_i \nabla v_i\|^2_{\partial \Omega_i}
\leq \sum_{i=1}^2 h \| n_i \|_{A_i}^{-1} \| n_i \|^2_{A_i,\partial \Omega_i} \| \nabla v_i \|^2_{A_i, \partial \Omega_i}
\\
& \leq \sum_{i=1}^2 h  \| \nabla v_i \|^2_{A_i, \partial \Omega_i}
\leq \sum_{i=1}^2  \| \nabla v_i \|^2_{A_i,\mcT_h{ \Omega_{h,i} }}
\leq \sum_{i=1}^2 \| \nabla v_i \|^2_{A_i, \Omega_{h,i}}
\leq \sum_{i=1}^2 \| \nabla v_i \|^2_{A_i, \Omega_{i}}
\end{align}
where we finally used the stability of the extended finite element space $V_h^E$ to pass from $\Omega_{h,i}$ to $\Omega_i$. For {B4} we 
need the Poincar\'e inequality 
\begin{equation}\label{eq:interface-poincare}
\sum_{i=1}^2 \| v_i \|^2_{\Omega_i} \lesssim \tn v \tn^2_{h} 
\end{equation}
To prove the Poincar\'e inequality we let $\phi \in H^1(\Omega)$ be the solution to (\ref{eq:interface-a})-(\ref{eq:interface-d}) 
with $f_i = v_i$. We then have 
\begin{align}
&\sum_{i=1}^2 \| v_i \|^2_{\Omega_i} 
= 
\sum_{i=1}^2 (v_i,-\nabla \cdot A_i \nabla \phi)_{\Omega_i}
= 
 \sum_{i=1}^2 (\nabla v_i, A_i \nabla \phi)_{\Omega_i} - (v_i - \langle v \rangle , n_i \cdot A_i \nabla \phi)_{\Gamma}
 \\
&= a_\Omega(v,\phi) - a_{\partial \Omega}(\phi,v) 
\lesssim 
\tn v \tn_\Omega \tn \phi \tn_\Omega + \tn \phi \tn_{\partial \Omega} \| v \|_b
\lesssim 
\tn v \tn_h ( \tn \phi \tn_\Omega^2 + \tn \phi \tn^2_{\partial \Omega} )^{1/2}
\end{align}
We close the argument by using a trace inequality 
\begin{equation}
 \tn \phi \tn_{\partial \Omega} \lesssim \sum_{i=1}^2 \| \phi \|_{H^2(\Omega_i)}
\end{equation} 
followed by elliptic regularity 
\begin{equation}
\sum_{i=1}^2 \| \phi_i \|^2_{H^2(\Omega_i)} \lesssim\sum_{i=1}^2 \| v_i \|^2_{\Omega_i}
\end{equation}
to conclude that (\ref{eq:interface-poincare}) holds. Finally, {B5} follows by inserting the exact solution into (\ref{eq:interface-method}) 
and using integration by parts. 

\paragraph{Error Estimate.} Finally, using the interpolation theory 
for $V_h^E$ combined with the abstract error estimate (\ref{eq:abs-best-approx}) we get
\begin{equation}
\tn u - u_h \tn_h \lesssim 
\tn u - \pi_h u \tn_h \lesssim h^k \| u \|_{H^k(\Omega)} 
\end{equation}
which for instance holds $C^0$ Lagrange elements of order $k$.

\subsection{Fourth Order Boundary Value Problem}
\paragraph{The Model Problem.}
Let $\Omega \subset  \IR^d$ be a domain with smooth boundary $\partial \Omega$.  Consider the biharmonic  problem 
\begin{alignat}{3}\label{eq:biha-a}
\Delta^2 u &=f &\qquad &\text{in $\Omega$}
\\ \label{eq:biha-b}
u=\nabla_n u & = 0 &\qquad &\text{on $\partial \Omega$}
\end{alignat}
Testing with $v \in H^2(\Omega)$ and integrating by parts we obtain the weak form 
\begin{align}
(f,v)_\Omega &= (\Delta^2 u, v)_\Omega = - (\nabla \Delta u, \nabla v)_{\partial \Omega} + (\nabla_n \Delta u, v)_{\partial \Omega} 
\\
&\qquad = (\Delta u, \Delta v)_{\partial \Omega} - (\Delta u, \nabla_n v)_{\partial \Omega} + (\nabla_n \Delta u, v)_{\partial \Omega} 
\end{align}
With $V= \{v \in H^2(\Omega) : \text{$v = \nabla_n v = 0$ on $\partial \Omega$} \}$ we get the weak statement: find $u \in V$ such that 
\begin{equation}\label{eq:biha-weak}
a(u,v) = l(v) \qquad \forall v \in V
\end{equation} 
where 
\begin{equation}
a(u,v) = (\Delta u, \Delta v)_\Omega, \qquad l(v) = (f,v)_\Omega
\end{equation}
We also note that for $v\in V$ we have 
\begin{equation}\label{eq:biha-H2-control}
\| v \|_{H^2(\Omega)} \lesssim \| \Delta v \|_\Omega
\end{equation}
To prove (\ref{eq:biha-H2-control}) we first use partial integration 
\begin{align}
(\Delta v, \Delta v)_\Omega &= - (\nabla v, \nabla \Delta v)_\Omega + \underbrace{(\nabla_n v, \Delta v)_{\partial \Omega}}_{=0}
= - (\nabla v, (\nabla^2 v)\cdot \nabla)_\Omega 
\\
&\qquad = (\nabla^2 v, \nabla^2 v)_\Omega  - \underbrace{(\nabla v, (\nabla^2 v)\cdot n)_{\partial \Omega}}_{=0}
= (\nabla^2 v, \nabla^2 v)_\Omega
\end{align}
since for $v \in V$ we have that the full gradient $\nabla v = 0$ on $\partial \Omega$. This fact follows by observing that the boundary 
$\partial \Omega$ is the zero levelset of $u$ and that the gradient is orthogonal to the levelsets of $u$. Therefore the tangential part of the 
gradient at the boundary is zero.
%
%
%
Then using a duality argument, similar to the verification of 
(\ref{eq:interface-poincare}), we can show that we have the Poincar\'e inequality 
\begin{equation}
\| v \|_\Omega \lesssim \|\Delta v \|_\Omega
\end{equation}
and finally we have 
\begin{equation}
\| \nabla v \|^2_\Omega = (\nabla v, \nabla v)_\Omega = -(v,\Delta v)_\Omega \leq \frac12 \| v \|^2_\Omega + \frac12 \| \Delta v \|^2_\Omega 
\end{equation}
This completes the verification of (\ref{eq:biha-H2-control}).

We finally conclude using Lax-Milgram that there is an exact solution $u\in V$ to (\ref{eq:biha-weak}).

\paragraph{The Finite Element Method.}

The finite element method takes the form: find $u_h \in V_h^E \subset H^2(\Omega)$, such that 
\begin{equation}\label{eq:biha-method}
a_h(u_h, v) = l_h(v) \qquad v \in V_h^E
\end{equation}
where the forms are 
\begin{align}
a_h(v,w) &= (\Delta v, \Delta w)_\Omega + (\Delta v, \nabla_n w)_{\partial \Omega} - (\nabla_n \Delta v,w)_{\partial \Omega}
\\
&\qquad + (\Delta w, \nabla_n v)_{\partial \Omega} - (\nabla_n \Delta w ,v)_{\partial \Omega} 
\\
&\qquad + \beta ( h^{-1}(  (\nabla_n v, \nabla_n w)_{\partial \Omega}  + \gamma (v,w)_{\partial \Omega} ) 
\\
l_h(v) &= (f,v)_{\partial \Omega}
\end{align}
 with $\beta$ and $\gamma$ positive parameters.

\paragraph{The Abstract Setting.} Let $V_h^E\subset H^2(\Omega)$ be an extended finite element space and define
\begin{align}
a_\Omega (v,w) &= (\Delta v, \Delta w)_{\Omega} 
\\
a_{\partial \Omega} (v,w) &= (\Delta v, \nabla_n w)_{\partial \Omega} - (\nabla_n \Delta v, w)_{\partial \Omega} 
\\
b(v,w) &=h^{-1} (\nabla_n v, \nabla_n w)_{\partial \Omega}  + \gamma h^{-3} (v,w)_{\partial \Omega}
\end{align}
and norms
\begin{align}
\tn v \tn^2_\Omega &= \| \Delta v \|^2_{\Omega}
\\
\tn v \tn^2_{\partial \Omega} &= h \|\Delta v\|^2_{\partial \Omega} + h^3\|\nabla_n \Delta v\|^2_{\partial \Omega} 
\\
\| v \|^2_b &=  h^{-1} \|\nabla_n v\|^2_{\partial \Omega}  + \gamma h^{-3} \|v\|^2_{\partial \Omega}
\end{align}
Next we verify the assumptions. 

\paragraph{Verification of Assumptions.} {B1} is trivial.  {B2} follows directly from the Cauchy-Schwarz inequality 
\begin{align}
a_{\partial \Omega} (v,w)  &= (\Delta v, \nabla_n w)_{\partial \Omega} - (\nabla_n \Delta v,w)_{\partial \Omega} 
\\
&\leq \|\Delta v\|_{\partial \Omega} \|\nabla_n w\|_{\partial \Omega} +  \|\nabla_n \Delta v\|_{\partial \Omega} \|w\|_{\partial \Omega} 
\\
&\leq ( h \|\Delta v\|^2_{\partial \Omega} +   h^3 \|\nabla_n \Delta v \|^2_{\partial \Omega} )^{1/2} 
(  h^{-1} \|\nabla_n w\|^2_{\partial \Omega} +  h^{-3} \|w\|^2_{\partial \Omega} )^{1/2}
\\
&\leq \tn v \tn_{\partial \Omega} \| w \|_b
\end{align}
{B3.} Using an inverse estimate to pass from $\partial \Omega$ to $\mcT_h(\partial \Omega)$, and an inverse inequality to 
remove $\nabla_n$ in the second term, and finally the stability (\ref{eq:E-stability}) of the extended finite element space we get
\begin{align}
\tn v \tn^2_{\partial \Omega} &= h \|\Delta v\|^2_{\partial \Omega} + h^3\|\nabla_n \Delta v\|^2_{\partial \Omega} 
\lesssim \|\Delta v\|^2_{\mcT_h(\partial \Omega)} + h^2\|\nabla_n \Delta v\|^2_{\mcT_h(\partial \Omega)}
\\
&\qquad  
\lesssim \|\Delta v\|^2_{\mcT_h(\partial \Omega)} +\|\Delta v\|^2_{\mcT_h(\partial \Omega)}
\lesssim \| \Delta v \|^2_{\Omega_h} 
\lesssim  \| \Delta v \|^2_{\Omega} 
=\tn v \tn_\Omega^2
\end{align}
{ B4.} Follows from a Poincar\'e inequality which we may derive using a duality argument. {B5.} Follows by inserting the exact solution 
into the method (\ref{eq:biha-method}) and using partial integration twice.

\paragraph{Error Estimate.}
Again using the interpolation theory 
for $V_h^E$ combined with the abstract error estimate (\ref{eq:abs-best-approx}) we get
\begin{equation}
\tn u - u_h \tn_h \lesssim 
\tn u - \pi_h u \tn_h \lesssim h^{k-2} \| u \|_{H^k(\Omega)} 
\end{equation}
which holds for $C^1$ elements of order $k$ such as tensor product hermite splines of order $k=3$ or the Argyris element of order $k=5$ 
on triangles in two dimensions.

\bigskip
\section{Numerical Examples}\label{sec:num}

In the numerical examples below, we use the following implementation of the extension operator. The mapping $S_h$ is constructed by 
associating with each element $T\in  \mcT_{h,B}$ the element in $\mcT_{h,I}$ which minimizes the distance 
between the element centroids. For each $x \in \mcX_h \setminus \mcX_{h,I}$ the weights in the nodal average 
$\langle \cdot \rangle_x$, see (\ref{eq:average-nodal}), is taken to be $1$ on precisely one element $T_x \in \mcT_h(x)$ and 
zero on all elements in $\mcT_h(x) \setminus T_x$, where we recall that $\mcT_h(x)$ is the set of elements which has $x$ as a vertex. Note that this choice of weights corresponds to simply 
defining the nodal value in $x \in \mcX_h \setminus \mcX_{h,I}$ by $((F_h v)|_{T_x})|_x$, where 
$F_h$ is defined in (\ref{eq:Fh}). This particular implementation has the advantage that it introduces relatively few non zero elements in the mass and stiffness matrix. 

In the examples below, the meshsize is defined by $h=1/\sqrt{\text{NNO}}$, where NNO denotes the number of corner nodes for the geometrical elements in the active mesh.

\subsection{Higher Order Approximation of a Poisson Boundary Value Problem}

On the disc $\Omega = \{ r: \; r < 0.5\}$, $r=\sqrt{x^2 +y^2}$, we consider a problem with manufactured solution
\begin{equation}
u =\cos{(\pi r)}
\end{equation}
corresponding to the right hand side
\begin{equation}
f=(\pi (\sin(\pi r) + \pi r\cos(\pi r))/r
\end{equation}
With this right hand side and $u=0$ on $\partial\Omega$, we solve (\ref{eq:poisson}) using triangular $P^2$ elements using linearly cut elements with boundary value correction \cite{BuHaLa18}. The Nitsche parameter was set to $\beta =10^{2}$. 

In Figure \ref{fig:poisson} we show the solution on a mesh in a sequence of halving the meshsize, and in Figure \ref{fig:convpoisson}
we show the observed convergence in $L_2(\Omega)$ and in $H^1(\Omega)$. The expected convergence of $O(h^3)$ is attained in $L_2$ and $O(h^2)$ in $H^1$.

\subsection{The Biharmonic Problem}

In this example we consider higher regularity tensor product Hermite splines to construct conforming approximations 
of the biharmonic and  the triharmonic problem. Nitsche's method was
used in the context of embedded boundaries and $C^1$-splines in
\cite{HS12}, but without treating the potential stability issues on
the cut boundary. Starting with the biharmonic problem we use $C^1$ tensor product Hermite 
splines as our conforming finite element space. We approximate the boundary by cubic $C^1$ splines and the cut geometry, 
which is used for quadrature, is then given by isoparametrically mapped $P^3$ triangles; more details can be found in \cite{BHL20a}.

The domain is here given by the disc \[\Omega = \{ r: \; r < r_0\},\quad \text{where $r=\sqrt{(x-1/2)^2 +(y-1/2)^2}$, $r_0=1/2$}\]
We use the manufactured solution 
$u=10^3 (r_0^2-r^2)^2/64$ corresponding to the right hand side $f=10^3$.
The boundary conditions are $u=0$ and $\nabla_n u=0$ on $\partial\Omega$ and we chose $\beta=100$, $\gamma=1$ in (\ref{eq:biha-method}).

In Figure \ref{fig:biharmonic} we show the solution on a mesh in a sequence of halving the meshsize, and in Figure \ref{fig:convbiharmonic}
we show the observed convergence in $L_2(\Omega)$ and in $H^1(\Omega)$. The expected convergence of $O(h^4)$ is attained in $L_2$, $O(h^3)$ in $H^1$, and $O(h^2)$ in $H^2$.


\subsection{A Poisson Interface Problem}

Here we use $P^1$ elements for an interface problem of the type (\ref{eq:interface-a})--(\ref{eq:interface-d}), but with boundary data given by the exact solution.
The domain inside the interface is 
\[\Omega_1 = \{ r: \; r < r_0\},\quad \text{where $r=\sqrt{(x-1/2)^2 +(y-1/2)^2}$, $r_0=1/4$.}\]
and the outer domain is $\Omega_2 = (0,1)\times(0,1)\setminus \bar{\Omega}_1$. We choose $A_1=5 I$ and $A_2= 2 I$, where $I$ is the identity matrix.
We use a fabricated solution
\begin{equation}
u=\left\{\begin{array}{>{\displaystyle}l}-(r^2/2-r_0^2/2+r_0^2/5) \quad \text{if $r>r0$}\\[4mm]
-r^2/5 \quad \text{if $r<r0$}  \end{array}\right.
\end{equation}
corresponding to a right hand side $f=4$. The Nitsche parameter was set to $\beta =10$ and the averaging weights in (\ref{eq:weights}) were set following \cite{HH02} (using area weighting).

In Figure \ref{fig:interface} we show the solution on a mesh in a sequence of halving the meshsize, and in Figure \ref{fig:convinterface}
we show the observed convergence in $L_2(\Omega)$ and in $H^1(\Omega)$. The expected convergence of $O(h^2)$ is attained in $L_2$ and $O(h)$ in $H^1$.

\subsection{Higher Order PDE}

We can easily extend the method to the triharmonic problem
\begin{equation}
- \Delta^m u = f \quad {\text{in $\Omega$}}
\end{equation}
Consider the case $m=3$, then we get 
\begin{align}
-(f,v)_\Omega &= (\Delta^3 u, v )_\Omega 
\\
&=   (\nabla_n \Delta^{2} u, v)_{\partial \Omega} -  (\nabla \Delta^{2}u, \nabla v ) _\Omega
\\
&=  (\nabla_n \Delta^{2} u, v)_{\partial \Omega} - ( \Delta^{2}u, \nabla_n v )_{\partial \Omega} + (\Delta^2 u, \Delta v )_\Omega
\\
&=  (\nabla_n \Delta^{2} u, v)_{\partial \Omega} - ( \Delta^{2}u, \nabla_n v ) _{\partial \Omega} 
\\
&\qquad + (\nabla_n \Delta u, \Delta v )_{\partial \Omega} - (\nabla \Delta  u, \nabla \Delta v )_\Omega
\end{align}
and we note that the strong conditions manufactured by the partial integration in this case are 
\begin{equation}
u = \nabla_n u = \Delta u = 0 \qquad \text{on $\partial \Omega$}
\end{equation}
The above partial integration formula can then directly be used to construct a Nitsche formulation that requires $V_h \subset H^3(\Omega)$, which 
means that the finite element space must be $C^2$. The Hermite splines are only available for odd polynomial order $p$ with reqularity $C^{p-2}$ 
and  therefore we use $p=5$, which are $C^3$, for the triharmonic  problem.

We consider a problem with constructed solution $u=c x^3y^3(x - 1)^3(y - 1)^3$ with $c=10^4$. We then construct the corresponding right-hand side for Poisson's problem, the biharmonic problem, and the triharmonic problem. We solve the problem on the domain $(0,1)\times (0,1)$ on a mesh covering a slightly larger domain $(-0.21,1.1)\times (-0.31,1.1)$. In all cases we set $\beta=10^3$. In Fig. \ref{fig:mesh} we show a mesh with the domain boundary indicated. In Fig. \ref{fig:3elev} we show elevations of the computed solutions on the same mesh, and in Fig. \ref{fig:convPBT} we show the energy convergence (convergence in $a_\Omega(u,u)$) for the three different problems.

\paragraph{Acknowledgements.}
This research was supported in part by the Swedish Research
Council Grants Nos.\  2013-4708, 2017-03911, 2018-05262, and the Swedish
Research Programme Essence. EB was supported in part by the EPSRC grant EP/P01576X/1.

\bigskip
\bigskip
\noindent
\footnotesize {\bf Authors' addresses:}

\smallskip
\noindent
Erik Burman,  \quad \hfill \addressuclshort\\
{\tt e.burman@ucl.ac.uk}

\smallskip
\noindent
Peter Hansbo,  \quad \hfill \addressjushort\\
{\tt peter.hansbo@ju.se}

\smallskip
\noindent
Mats G. Larson,  \quad \hfill \addressumushort\\
{\tt mats.larson@umu.se}

\bibliographystyle{abbrv}
\footnotesize{
\bibliography{ref}
}

\newpage

\begin{figure}[ht]
	\begin{center}
		\includegraphics[scale=0.20]{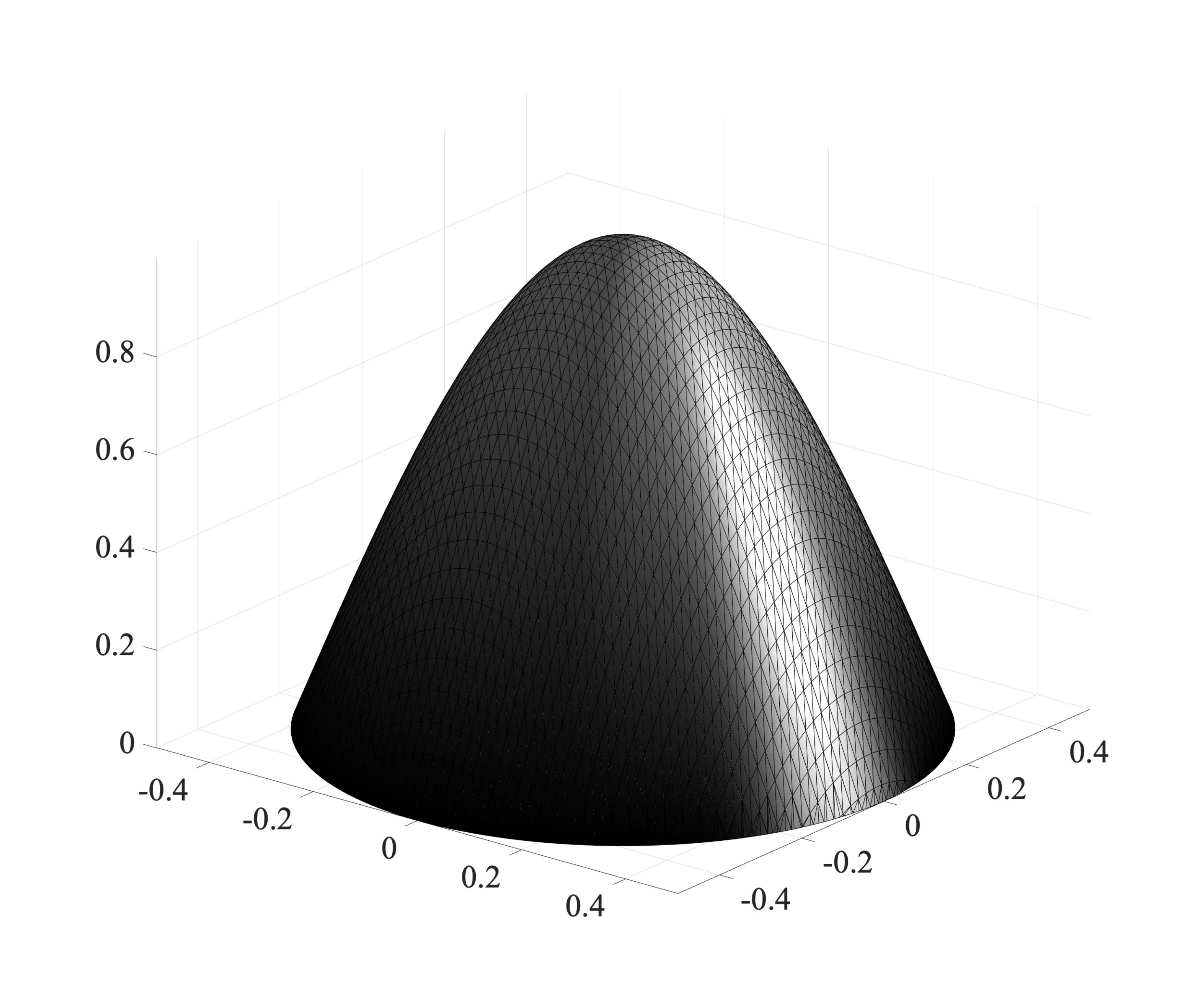}
	\end{center}
	\caption{Elevation of the computed Poisson solution on a particular mesh.}
\label{fig:poisson}
\end{figure}

\begin{figure}[ht]
	\begin{center}
		\includegraphics[scale=0.25]{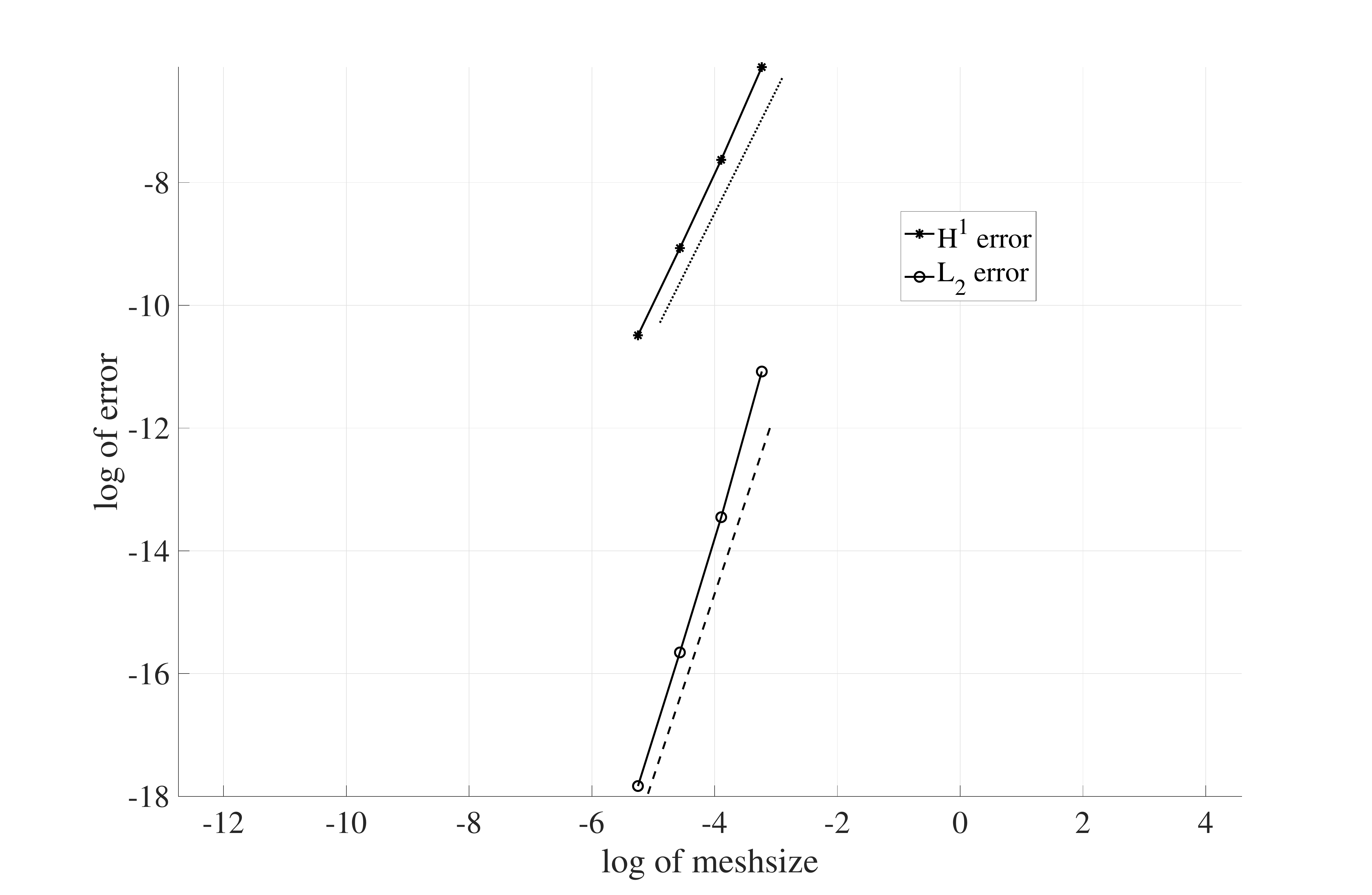}
	\end{center}
	\caption{Convergence for the Poisson problem. Dotted line has inclination $2:1$, dashed line has inclination $3:1$.}
	\label{fig:convpoisson}
\end{figure}

\begin{figure}[ht]
	\begin{center}
		\includegraphics[scale=0.20]{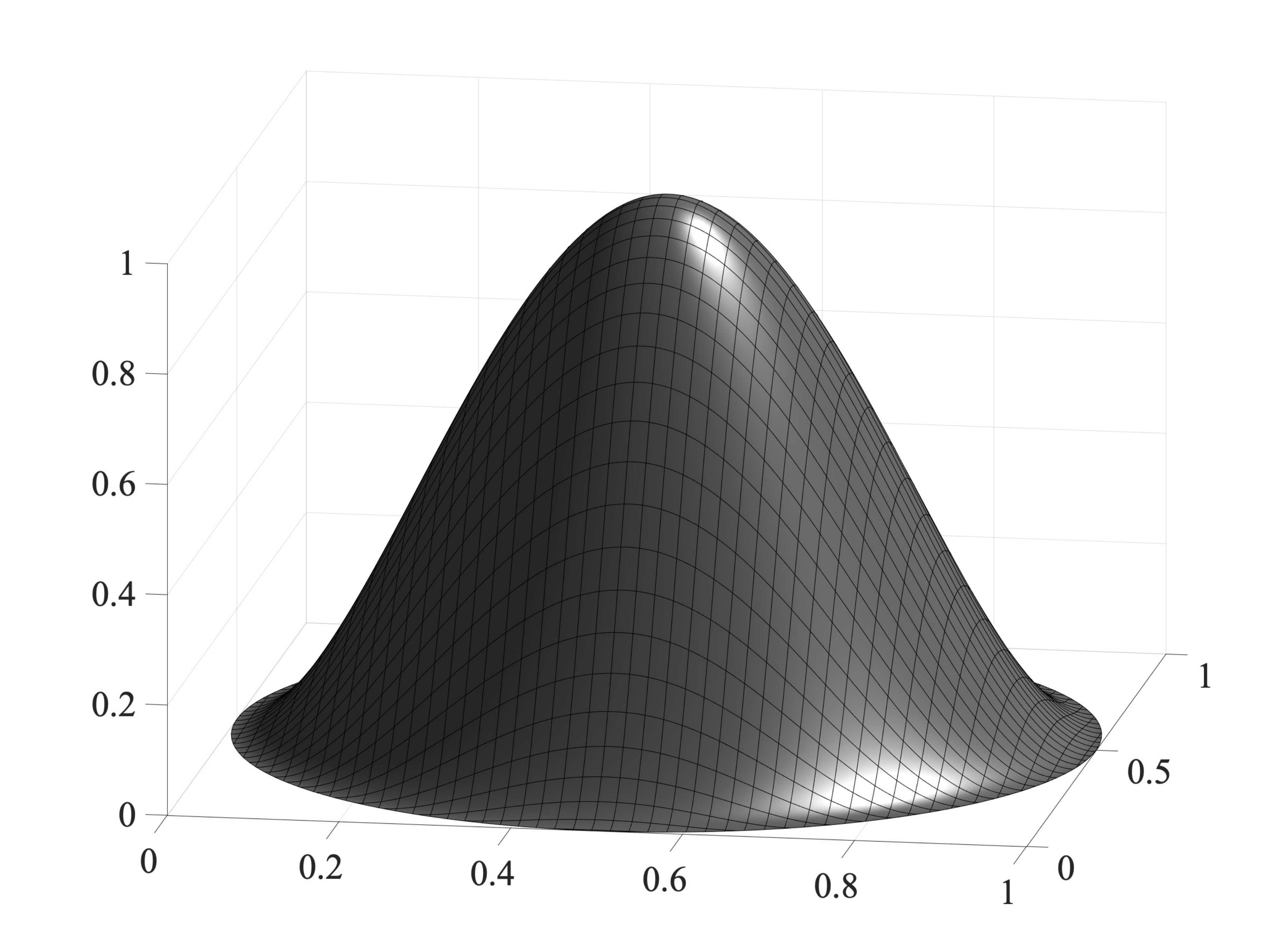}
	\end{center}
	\caption{Elevation of the computed biharmonic solution on a particular mesh.}
\label{fig:biharmonic}
\end{figure}

\begin{figure}[ht]
	\begin{center}
		\includegraphics[scale=0.30]{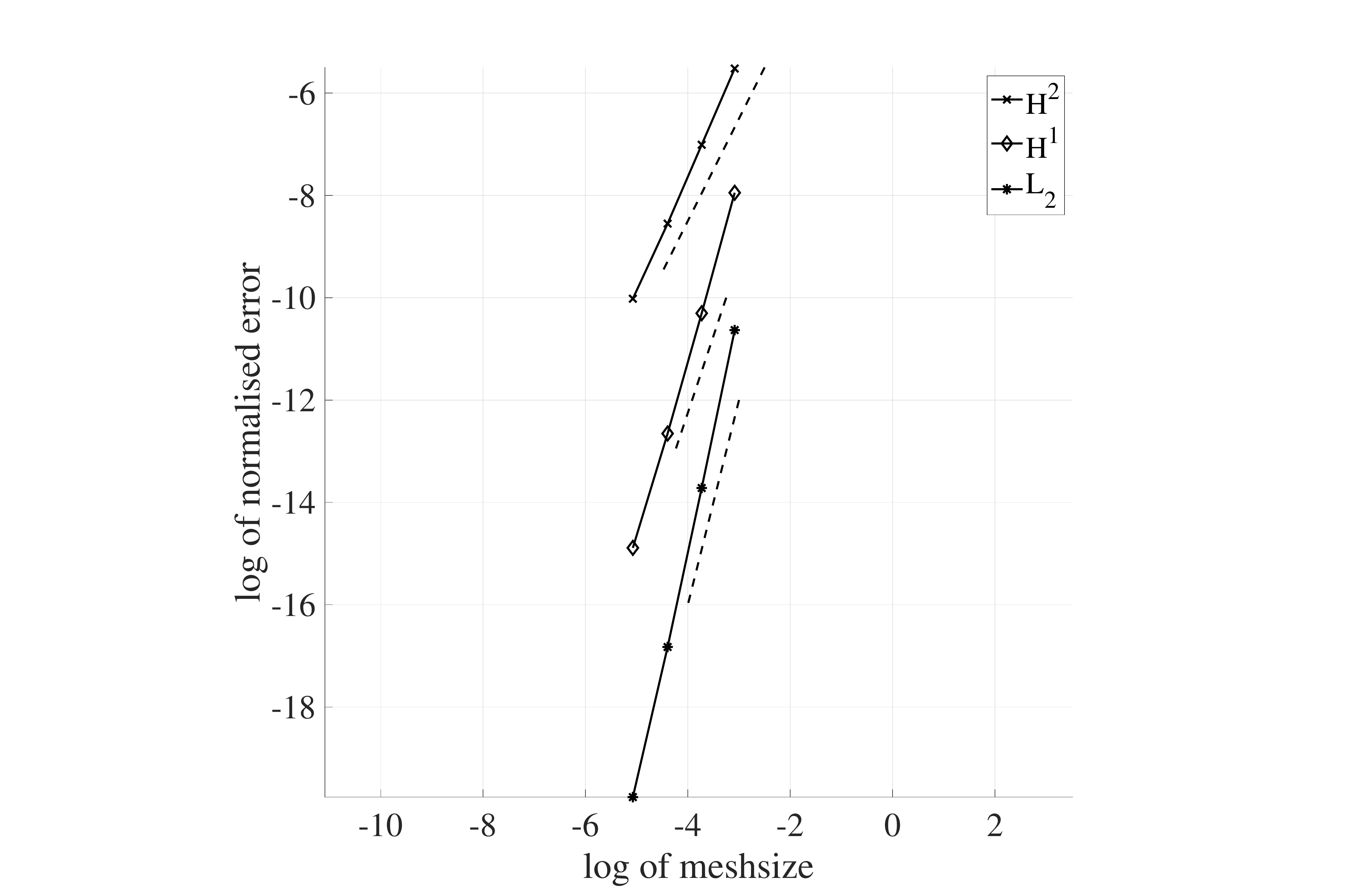}
	\end{center}
	\caption{Convergence for the biharmonic problem. Dashed lines have inclination 2:1, 3:1, and 4:1 from top.}
	\label{fig:convbiharmonic}
\end{figure}

\begin{figure}[ht]
	\begin{center}
		\includegraphics[scale=0.20]{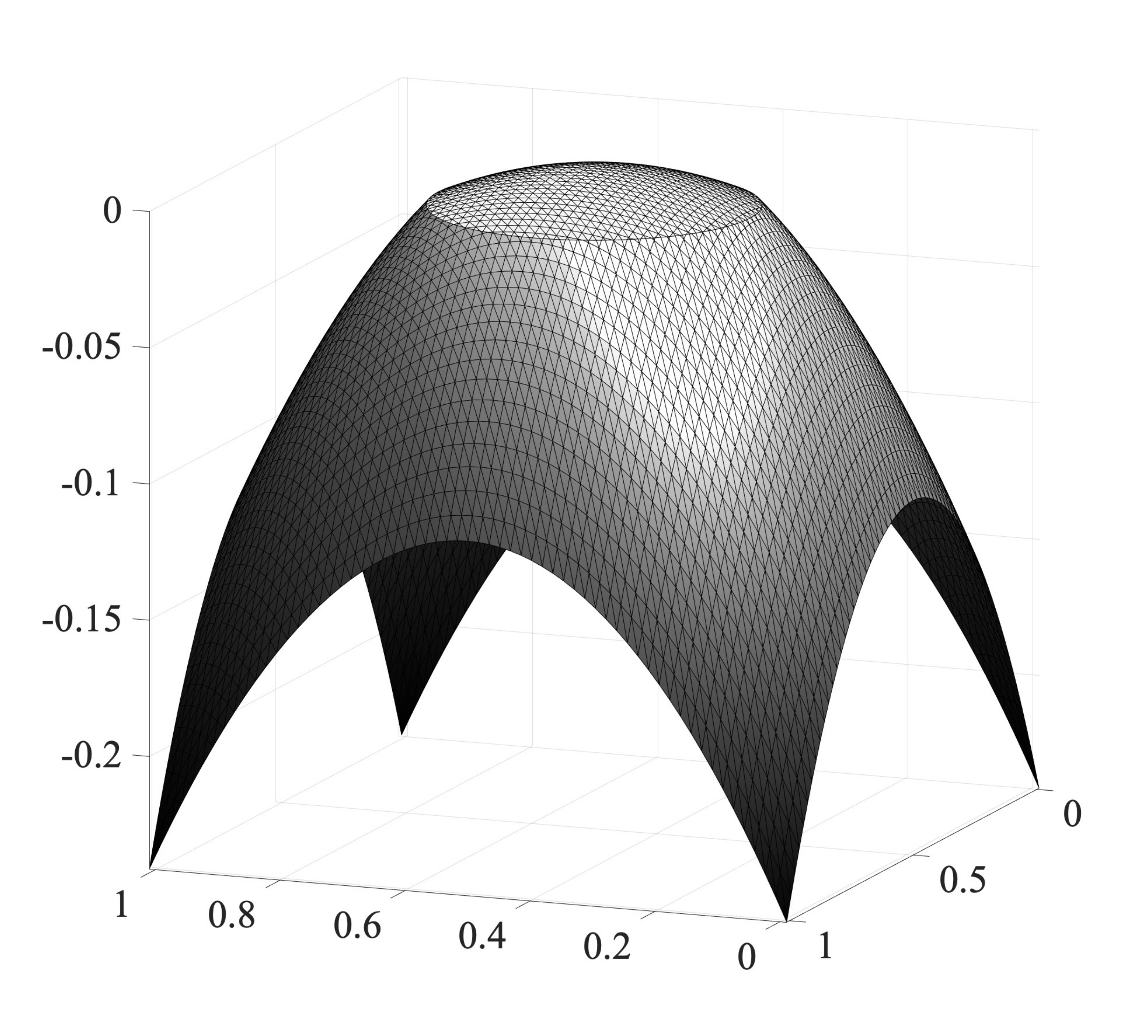}
	\end{center}
	\caption{Elevation of the computed interface solution on a particular mesh.}
\label{fig:interface}
\end{figure}

\begin{figure}[ht]
	\begin{center}
		\includegraphics[scale=0.30]{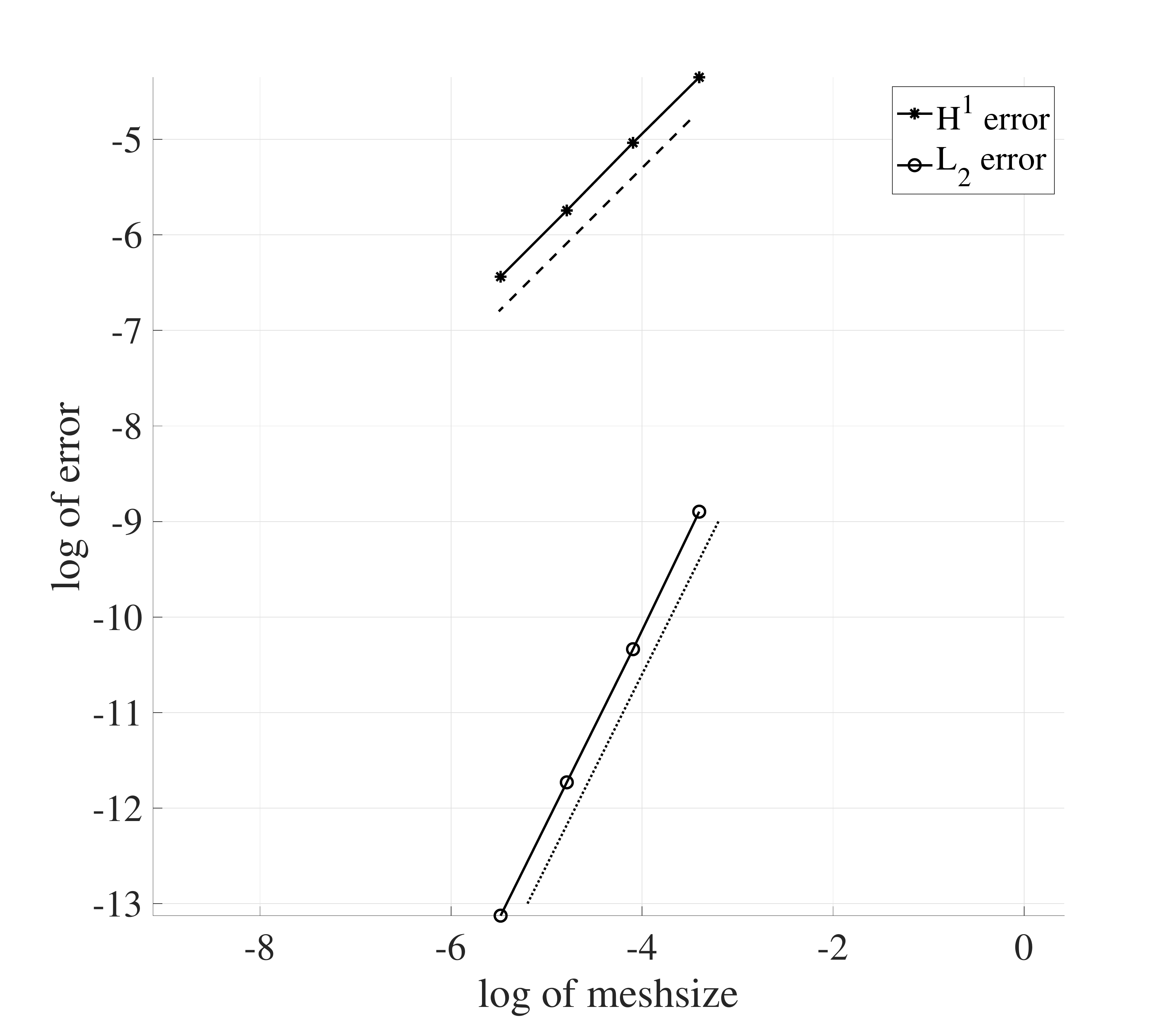}
	\end{center}
	\caption{Convergence for the interface problem. Dashed line has inclination $1:1$ and dotted line has inclination 2:1.}
	\label{fig:convinterface}
\end{figure}

\begin{figure}[ht]
	\begin{center}
		\includegraphics[scale=0.30]{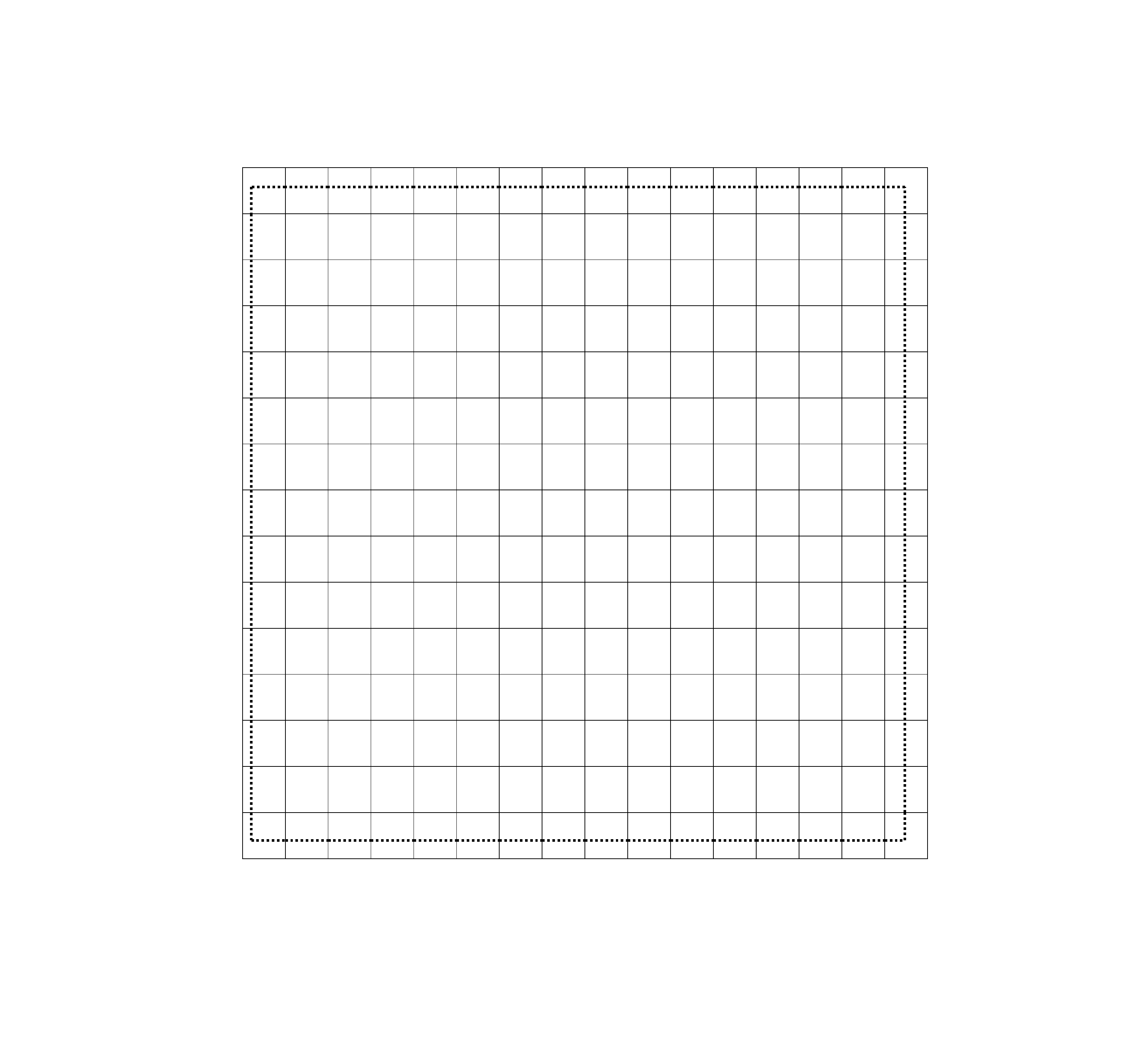}
	\end{center}
	\caption{Mesh used for the $C^2$ approximation. The boundary of the domain is dotted.}
	\label{fig:mesh}
\end{figure}
\begin{figure}[ht]
	\begin{center}
		\includegraphics[scale=0.20]{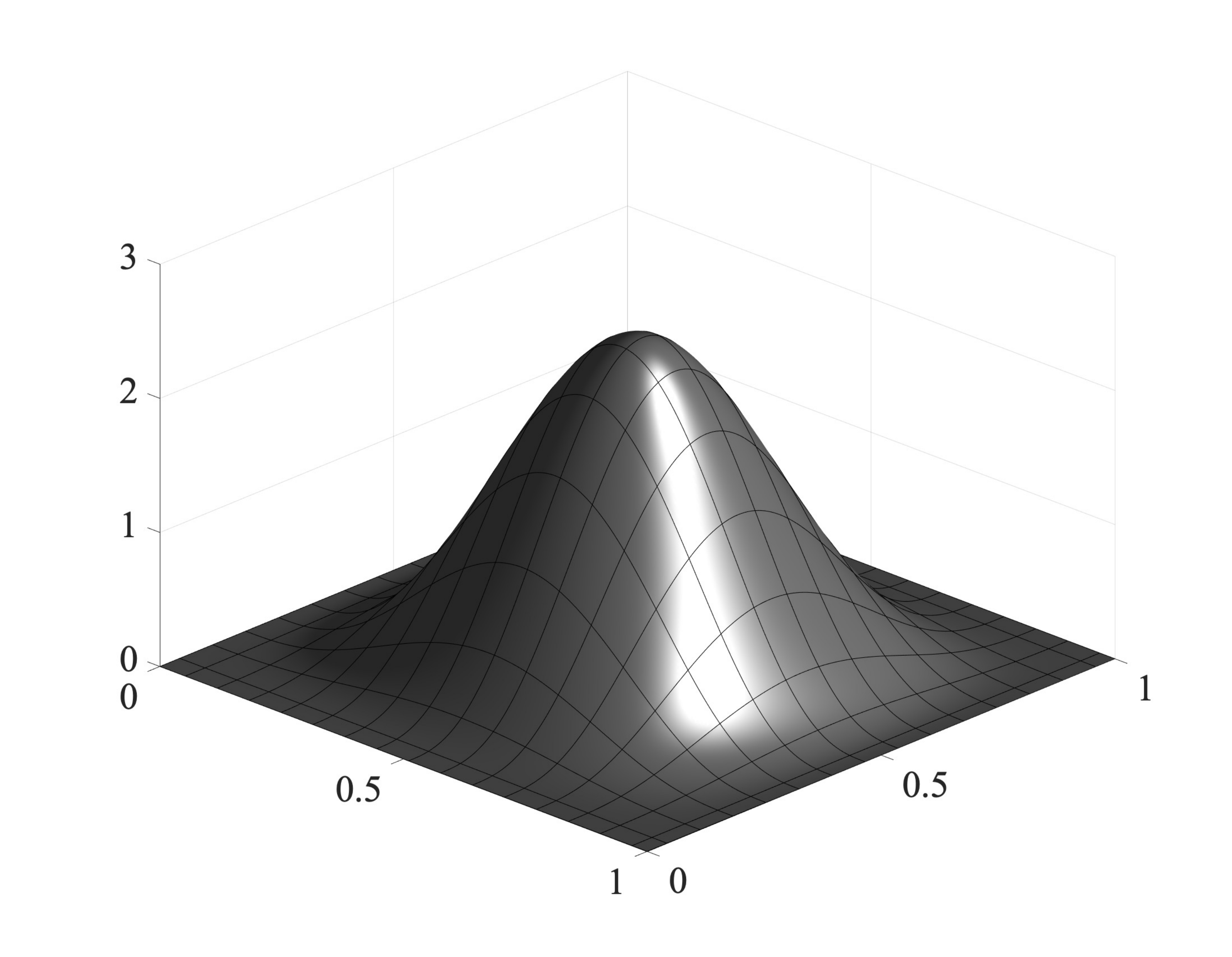}
		\includegraphics[scale=0.20]{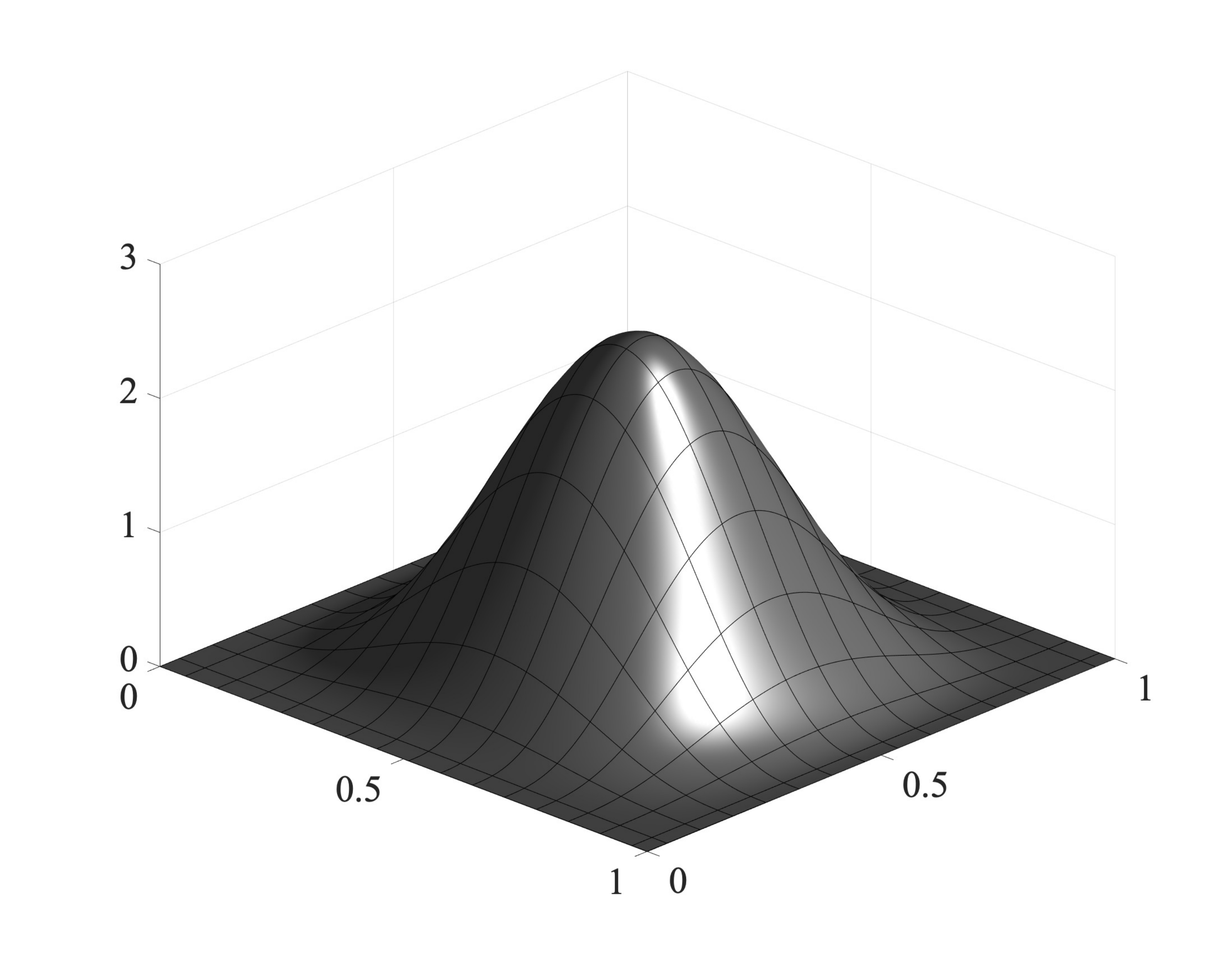}
		\newline\includegraphics[scale=0.20]{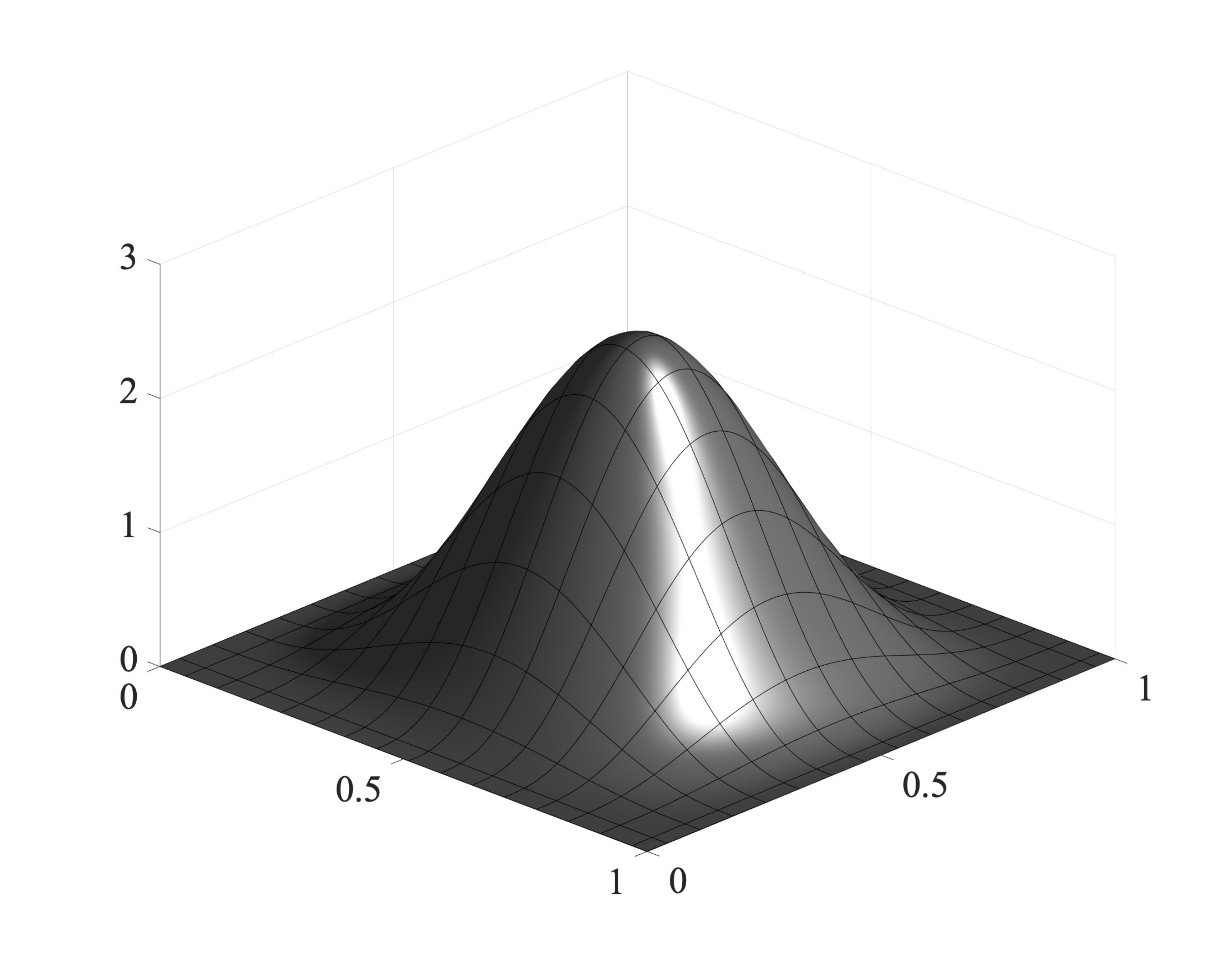}
	\end{center}
	\caption{Elevation of computed solutions for the Poisson, biharmonic, and triharmonic problems.}
	\label{fig:3elev}
\end{figure}
\begin{figure}[ht]
	\begin{center}
		\includegraphics[scale=0.30]{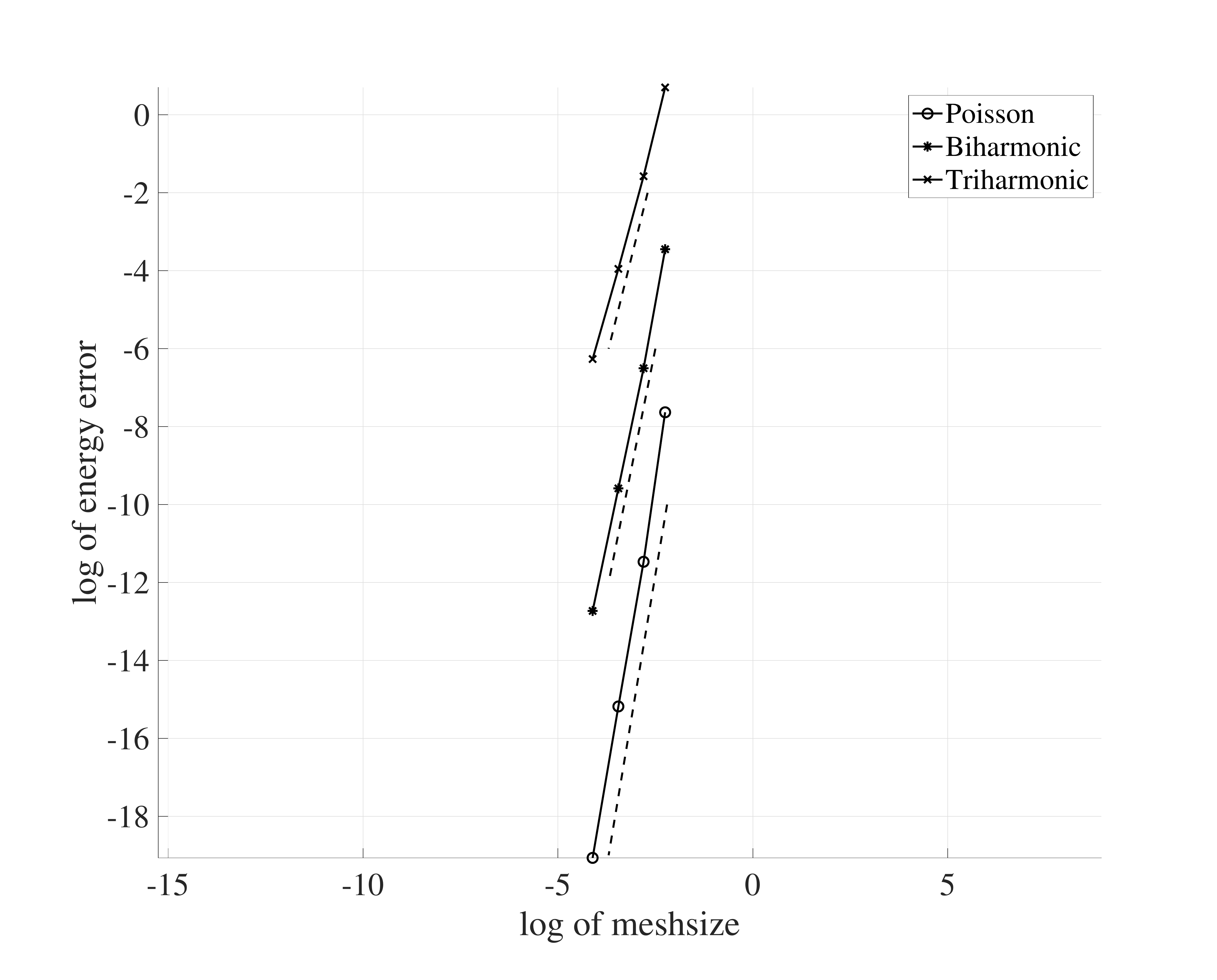}
	\end{center}
	\caption{Energy convergence for the  the Poisson, biharmonic, and triharmonic problems. Dashed lines have inclination 4:1, 5:1, and 6:1 from top.}
	\label{fig:convPBT}
\end{figure}

\end{document}